\documentclass[11pt]{amsart}
\pdfoutput=1 
\usepackage{graphicx}
\usepackage{amsmath, amsthm, amssymb, verbatim, subfig}
\usepackage{mathpazo}
\usepackage{fullpage}

  \newcommand{\C}[0]{\mathbb{C}}
\newcommand{\R}[0]{\mathbb{R}}  
  
\newcommand{\HH}[0]{{\mathbb{H} ^2}}
\newcommand{\calH}[0]{\mathcal{H}} \newcommand{\calL}[0]{\mathcal{L}} \newcommand{\calJ}[0]{\mathcal{J}}

\newcommand{\tensor}[0]{\otimes}
\newcommand{\zbar}[0]{\overline{z}}

\newcommand{\suchthat}{\;\ifnum\currentgrouptype=16 \middle\fi|\;}

\newtheorem{theorem}{Theorem}[section]
\newtheorem{lemma}[theorem]{Lemma}
\newtheorem{proposition}[theorem]{Proposition}
\newtheorem{corollary}[theorem]{Corollary}

\newtheorem{conjecture}[theorem]{Conjecture}
\newtheorem{question}[theorem]{Question}

\begin{document}

\title{Harmonic maps of punctured surfaces to the hyperbolic plane}
\author{Andy C. Huang}
\thanks{The author acknowledges support from U.S. National Science Foundation grants DMS 1107452, 1107263, 1107367 "RNMS: Geometric Structures and Representation Varieties" (the GEAR Network)}
\address{Rice University, Houston, TX USA}
\email{andy.c.huang@outlook.com}
\urladdr{http://www.math.rice.edu/~ach3/}
\allowdisplaybreaks

\begin{abstract}
In this paper, we construct {\em polynomial growth} harmonic maps from once-punctured Riemann surfaces of any finite genus to any even-sided, regular, ideal polygon in the hyperbolic plane. We also establish their uniqueness within a class of maps which differ by exponentially decaying variations. Previously, harmonic maps from once-punctured spheres to $\HH$ have been parameterized by holomorphic quadratic differentials on $\C$ (\cite{WA94}, \cite{HTTW95}, \cite{ST02}). Our harmonic maps, mapping a genus $g>1$ domain to a $k$-sided polygon, correspond to meromorphic quadratic differentials having one pole of order $(k+2)$ and $(4g +k-2)$ zeros (counting multiplicity). In this way, we can associate to these maps a holomorphic quadratic differential on the punctured Riemann surface domain. As an example, we specialize our theorems to obtain a harmonic map from a punctured square torus to an ideal square, and deduce the five possibilities for the divisor of its Hopf differential.
\end{abstract}

\maketitle

\tableofcontents

\section{Introduction and Main results}
This paper begins an investigation of the question: what is the shape of a harmonic map between surfaces from higher topological complexity to lower topological complexity? In order to focus this question, we first consider the special class of harmonic maps between compact hyperbolic surfaces of different genera. This family of maps has the appealing property that harmonic maps between surfaces are in bijection with homomorphisms between their fundamental groups.

For, any smooth map $u\in \mathcal{C}^{\infty}(\Sigma_g, \Sigma _h)$ between compact hyperbolic surfaces $(\Sigma _g, \sigma)$ and $(\Sigma _h, \rho)$ of any genera induces a homomorphism $u_*:\pi_1(\Sigma _g) \rightarrow \pi _1(\Sigma _h)$ between their fundamental groups. Conversely, since compact surfaces are $K(\pi_1,1)$ spaces, any homomorphism $\mu$ between their fundamental groups can be induced by a continuous map $f$ between them, i.e., $f_* \equiv \mu$. Negative curvature of the target metric $\rho$ ensures a unique harmonic representative exists in the homotopy class of that continuous map, if the map is non-constant (\cite{ES64}, \cite{Hartman67}). In the case the map is constant, the trivial homomorphism $u_*\left(\pi _1 (\Sigma _g )\right)=e$ is induced. We state this correspondence as
$$ \{ \left[u\right] | u\in \mathcal{C}^{\infty}(\Sigma_g, \Sigma _h), \mbox{ harmonic, non-constant}\} \overset{1-1}\longleftrightarrow \left\{ u_* \in Hom\left( \pi _1 (\Sigma _g) ,\pi_1 (\Sigma _h)\right) \setminus \{ e \} \right\}.$$

Furthermore, any smooth map $u$ defines a symmetric $2$-tensor $u^*(\rho)$ on $\Sigma _g$ by pulling back the target metric. We use the complex structure induced by $\sigma$ on $\Sigma _g$ to extract the $(2,0)$-component of this pullback metric, called the \emph{Hopf differential of $u$}, $$\Phi ^u:= \left(u^*(\rho)\right)^{(2,0)}.$$ Hopf observed that the harmonicity of $u$ implies holomorphicity of $\Phi ^u$ with respect to this complex structure \cite{Hopf51}. Thus, there is a map
$$\mathcal{C}^{\infty}(\Sigma _g, \Sigma _h) \ni u \longmapsto \Phi ^u \in QD(\Sigma _g),$$
where $QD(\Sigma _g)$ is the vector space of holomorphic quadratic differentials on $\Sigma _g$.

In this way, the complex analytic object $\Phi ^u$ captures some of the differential topology of the harmonic map $u$ and equivalently the algebraic information of the homomorphism $u_*$ between their fundamental groups $\pi_1(\Sigma _g)$ and $\pi_1(\Sigma _h)$. That $\Phi ^u$ captures some differential topology of the map is not a new concept. The dual measured foliations of $\Phi ^u$ describe the maximal and minimal stretch directions of the harmonic map $u$ (\cite{Minsky92}, \cite{HTTW95}).

It is our present aim to create a local model to study the handle-crushing harmonic maps from higher genus to lower genus surfaces, and then characterize their Hopf differentials. Such handle crushing maps can be locally described as mapping a surface of positive genus and with boundary onto a disk, non-injectively taking interior to interior and monotonically taking boundary to boundary. These induce the trivial homomorphism on fundamental groups of punctured surfaces. We produce a model of handle crushing harmonic maps in the following:

\begin{theorem} \label{thm:crusher_existence} {\emph (Existence)}
For any once punctured genus $g \geq 0$ surface $\Sigma _g$ and for any regular ideal $k$-sided polygon $\mathcal{P}_k$ in $\HH$ for $k>2$ even, there exists a harmonic map $h: \Sigma _g \rightarrow \HH$ whose image $h (\Sigma _g)$ is the interior of $\mathcal{P}_k$ and whose closure $\overline{h(\Sigma_g)}=\overline{\mathcal{P}_k}$. For this harmonic map, the Hopf differential has $(4g+k-2)$ zeros and one pole of order $(k+2)$ at the puncture.
\end{theorem}

\noindent Here, we say $\mathcal{P}_k$ is regular if it has a dihedral isometry group $D_k$ of order $2k$. We remark that the $g=0$ case of this theorem concerns harmonic maps $u:\C \rightarrow \HH$, and has already been established in \cite{HTTW95} for a more general target - any ideal polygon in $\HH$. Our Theorem \ref{thm:crusher_existence} generalizes to the $g>0$ domains from which a harmonic map to an ideal hyperbolic polygon can be found.

Our proof of Theorem \ref{thm:crusher_existence} is by approximation on compacta of a punctured Riemann surface. The evenness of $k$ ensures a convenient reflectional symmetry. To obtain convergence of the approximating sequence, we introduce and analyze \emph{parachute maps}, each of which is a harmonic map from a compact cylinder and each of which solves a toy harmonic mapping problem constrained by a partially free boundary condition. The regularity of $\mathcal{P}_k$ provides rotational symmetries which simplifies the analysis of the parachute maps' Hopf differentials.

The key techniques in our proof lie in the comparison of the energy of a harmonic map on a compactum of the punctured surface to the energy of a parachute map on the appropriate cylinder. We exploit the non-injectivity of the parachute map to bound the image distance in terms of the pointwise norm of its Hopf differential. We bound the Hopf differential, in turn, through studying its Laurent series.

As a by-product of this comparison to the parachute maps, we are also able to show that the energy densities for the harmonic maps from Theorem \ref{thm:crusher_existence} have polynomial growth (whose degrees determine the finite order pole and the number of sides of the image set $\mathcal{P}_k$). Furthermore, recognizing the parabolic nature (in the complex function theoretic sense) of the punctured Riemann surface domains, we also address uniqueness for these harmonic maps:

\begin{theorem} \label{thm:crusher_uniqueness} {\emph (Uniqueness)}
Suppose $h,v: \Sigma _g \rightarrow \HH$ are two harmonic maps. Denote their pointwise distance function by $d(z) := dist_{\HH} \left(h(z),v(z)\right)$. If we have $$\left(cosh\circ d\right) -1 \in \mathcal{L} ^p (\Sigma _g)$$ for some some $p\in (1, +\infty]$, then $d(z) \equiv 0$, i.e., the maps $v$ and $h$ must agree pointwise.
\end{theorem}

As an example application, we specialize Theorems \ref{thm:crusher_existence} and \ref{thm:crusher_uniqueness} for a square torus mapped to an ideal square. From the uniqueness provided by Theorem \ref{thm:crusher_uniqueness} and the large amount of symmetry, we are able to deduce the possibilities for the Hopf differential of the resulting harmonic map. In particular, we show that:

\begin{corollary}\label{cor:torus_crusher}
There exists a harmonic map $h$ from the punctured square torus $\Sigma _1$ to the ideal square $\mathcal{P}_4$ in $\HH$. Its Hopf differential $\Phi ^h$ has exactly one pole of order $6$ and has one of the the following three multiplicities of zeros, as depicted by the five arrangements in Figure \ref{fig:torus-to-square-hopf-possibilities}:
\begin{enumerate}
\item[(a)] three zeros of order $2$,
\item[(b)] or one zero of order $2$ and four of order $1$,
\item[(c)] or one zero of order $6$.
\end{enumerate}

\begin{figure}
  \centering
  \subfloat[]{\includegraphics[scale=.33]{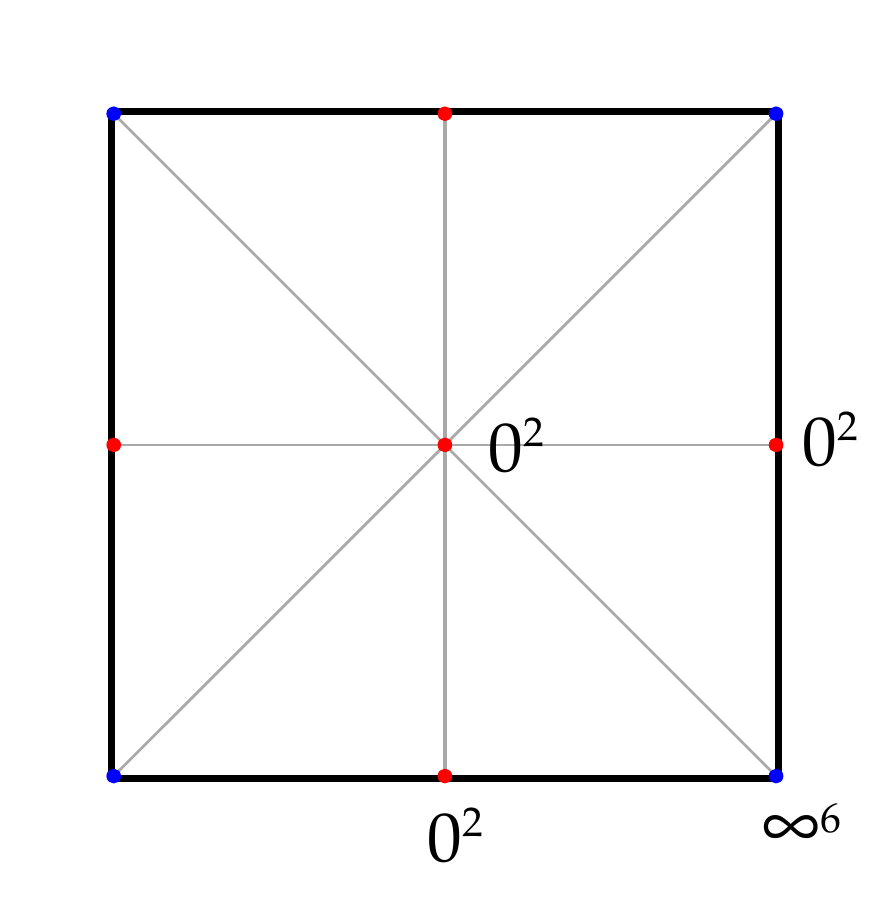}} \hskip 14pt
  \subfloat[(three configurations for a divisor with a zero of order $2$ and four simple zeros)]{\includegraphics[scale=.33]{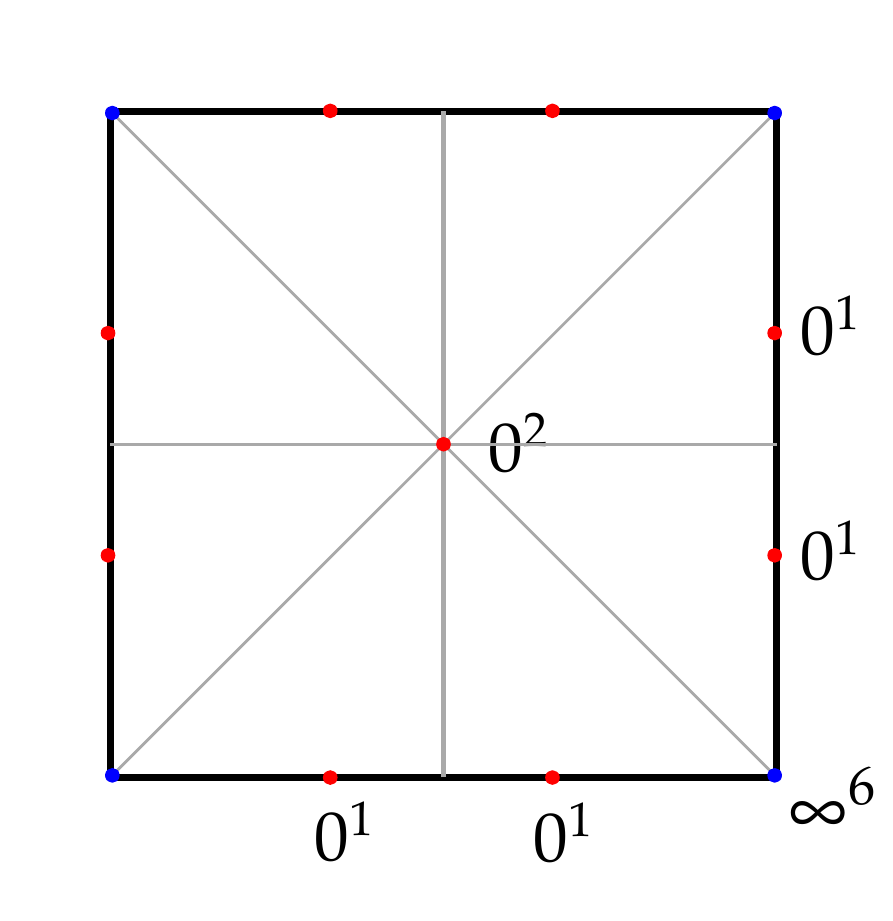}
              \includegraphics[scale=.33]{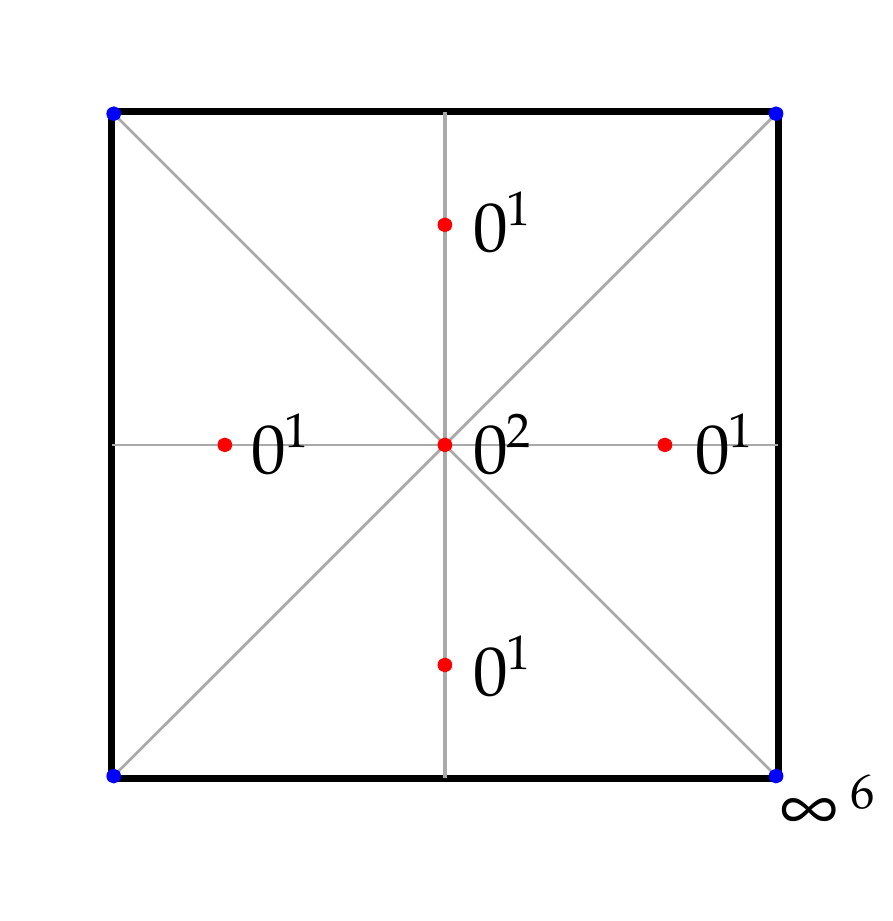}
              \includegraphics[scale=.33]{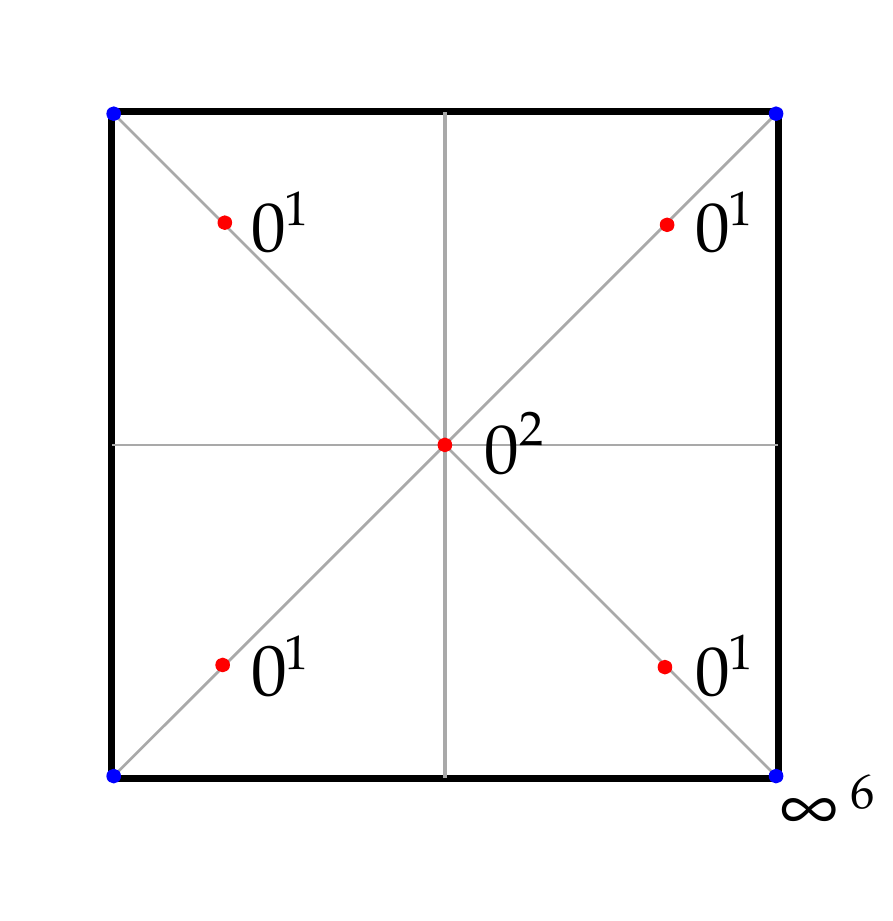}} \hskip 14pt
  \subfloat[]{\includegraphics[scale=.33]{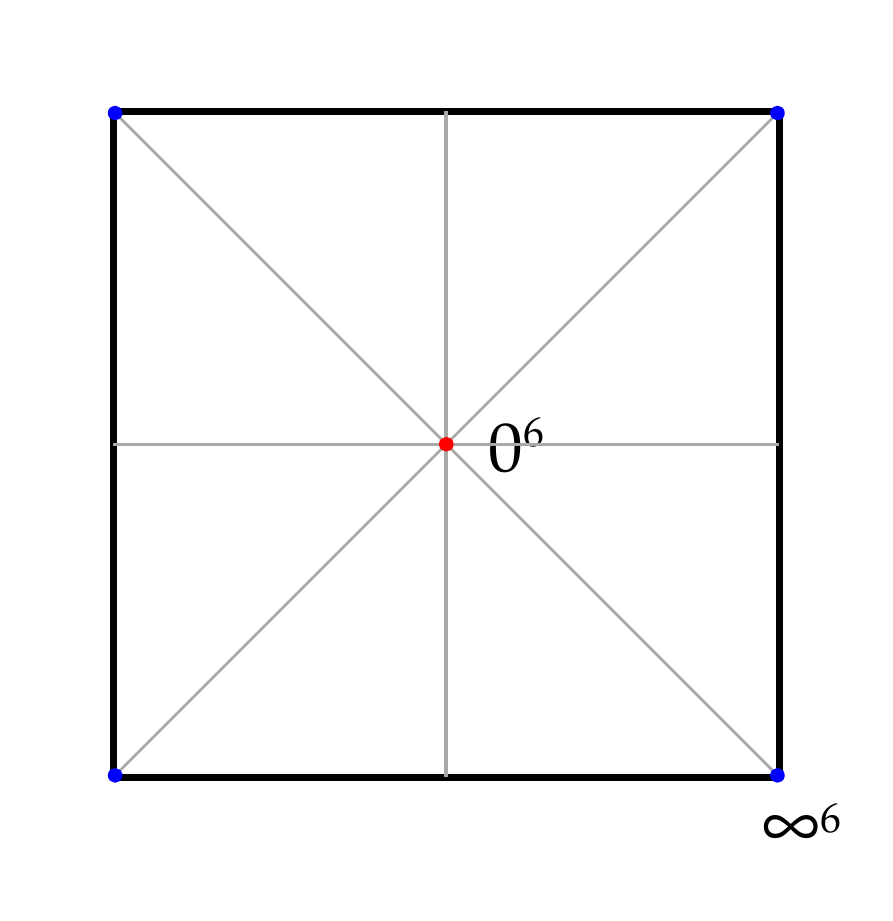}}
  \caption{\label{fig:torus-to-square-hopf-possibilities}The possible Hopf differential configurations for Corollary \ref{cor:torus_crusher}}
\end{figure}

In arrangement $(a)$, the Hopf differential is a square of a holomorphic one-form. Furthermore, two of these double zeros lie an a subset where $\mathcal{J} >0$ and orientation is preserved, while the third double zero lies on a subset where $\mathcal{J} <0$ and the orientation is reversed. Nonetheless, in the arrangements described by $(b)$ and $(c)$, the Hopf differential is {\em not} the square of a one-form.
\end{corollary}

By the square punctured torus, we mean the conformally unique punctured torus with a dihedral $D_4$ symmetry group of order $8$ which fixes the puncture. The Scherk map is the harmonic map $w:\C \rightarrow \HH$ with Hopf differential $\Phi^w (z) = z^2 dz^2$ and whose image lies in an ideal square. It is described in more detail in Theorem \ref{thm:scherkmap}. For visualization, it is helpful to realize it as the projection of the Scherk type minimal surface in $\HH \times \R$, discovered by Nelli and Rosenberg, onto $\HH$ (c.f. section 4 of \cite{NR02}). We approximate it numerically in Figure \ref{fig:scherksurface}.

\begin{figure}
\begin{center}
\includegraphics[scale=.3]{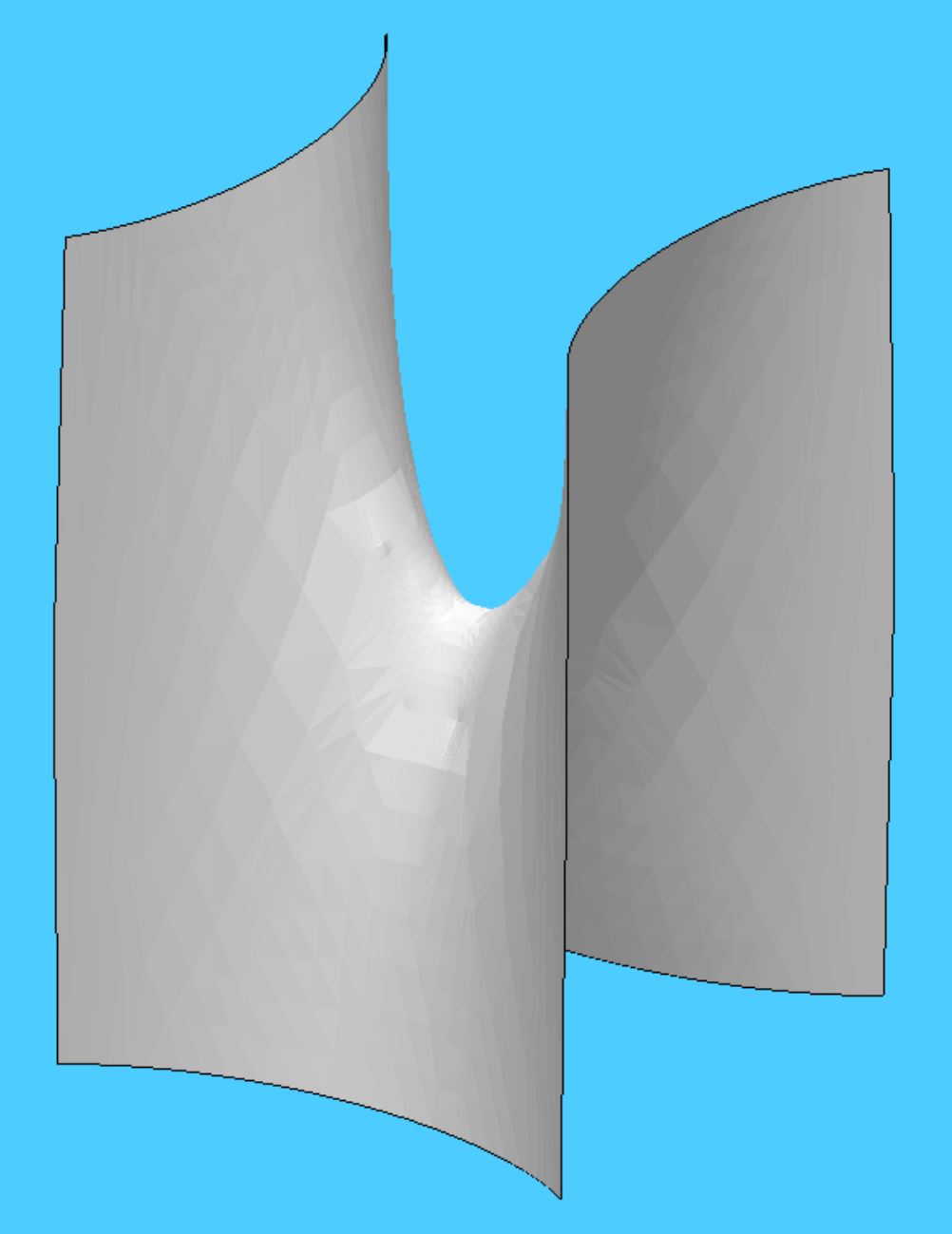}
\caption{\label{fig:scherksurface}Suspension of $u: \C\rightarrow \HH$ with $\Phi ^u (z) = z^2 dz^2$.\protect\footnotemark}
\end{center}
\end{figure}
\footnotetext{Produced using Ken Brakke's \emph{Surface Evolver} program.}

From only the Existence and Uniqueness Theorems \ref{thm:crusher_existence} and \ref{thm:crusher_uniqueness}, we will not be able to conclude precisely which of these three configurations is the unique Hopf differential realized. Nonetheless, we conjecture:
\begin{conjecture}
The Hopf differential $\Phi ^h$ of Corollary \ref{cor:torus_crusher} has divisor data described by arrangement (a).
\end{conjecture}
Observe that this configuration has a pole of order 6 at the puncture, fixed by the $D_4$ symmetry of the square torus, and three zeros of order 2, each on a line of reflectional symmetry. Thus, we can express $\Phi ^h$ in terms of the Weierstrass $\wp$ and $\wp '$ functions, and in fact as a square of a holomorphic $1$-form. Then, we have:

\begin{corollary}\label{cor:chen-gack-H2R}
If the Hopf differential $\Phi^h$ of the map $h:\Sigma _1 \rightarrow \HH$ has divisor data as in configuration (a) described in Corollary \ref{cor:torus_crusher}, then there exists a harmonic function $f:\Sigma _1 \rightarrow \R$ for which $\Phi ^h =- \Phi ^f $. In particular, this provides a minimal suspension $\tilde{h}:=(h, f):\Sigma _1 \rightarrow \HH \times \R$ with $\tilde{h}(\Sigma _1)$ a genus one minimal surface in $\HH \times \R$.
\end{corollary}

The construction of a \emph{minimal suspension} $\tilde{h}$ is directly adapted from \cite{Wolf98}. We numerically approximate a compact subset of $\tilde{h}(\Sigma _1)$ in Figure \ref{fig:chen-gack-H2R}. A depiction of the horizontal and vertical foliations is produced in Figure \ref{fig:foliation-torus_crusher_a}. This approximation of the minimal suspension provides numerical evidence for configuration (a).

\begin{figure}
  \centering
  {\includegraphics[scale=.3]{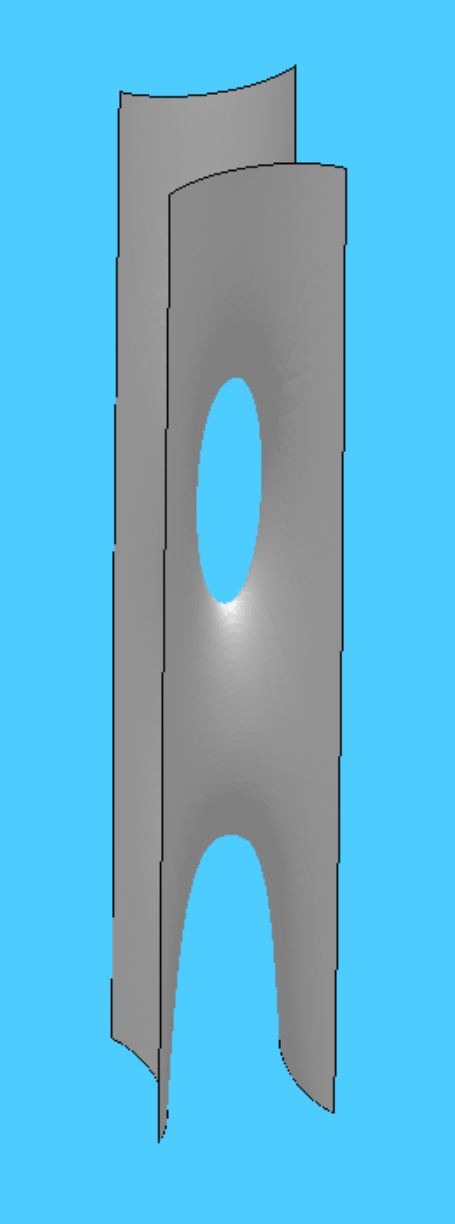}}\caption{\label{fig:chen-gack-H2R}
  Approximate minimal suspension of map from Corollary \ref{cor:chen-gack-H2R}.\protect\footnotemark}
\end{figure}
\footnotetext{Produced using Ken Brakke's \emph{Surface Evolver} program.}

\begin{figure}
  \centering
  \subfloat [$\mathcal{F}_\mathcal{H}$]
  {\includegraphics[scale=.25]{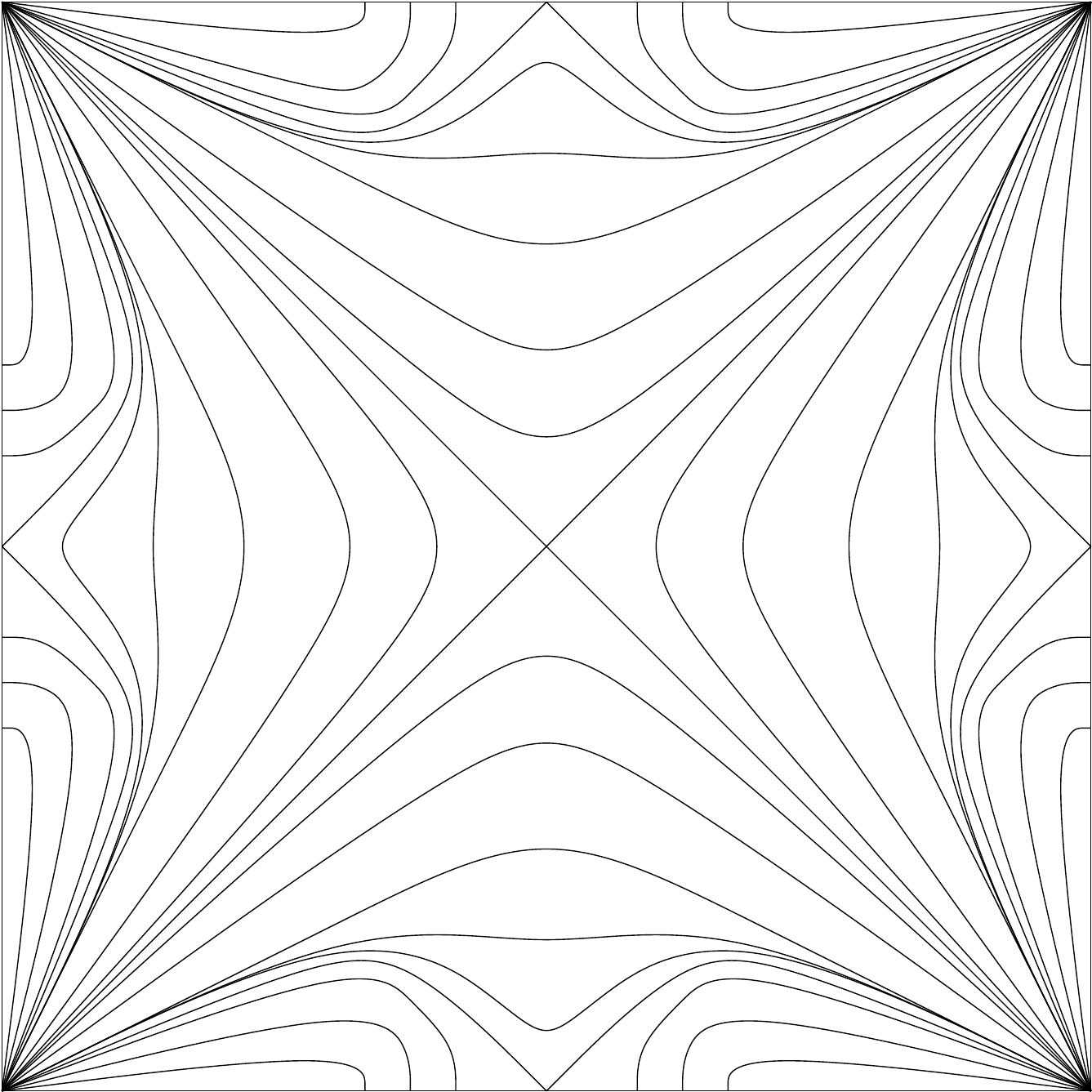}} \hskip 36pt
  \subfloat[$\mathcal{F}_\mathcal{V}$]
  {\includegraphics[scale=.25]{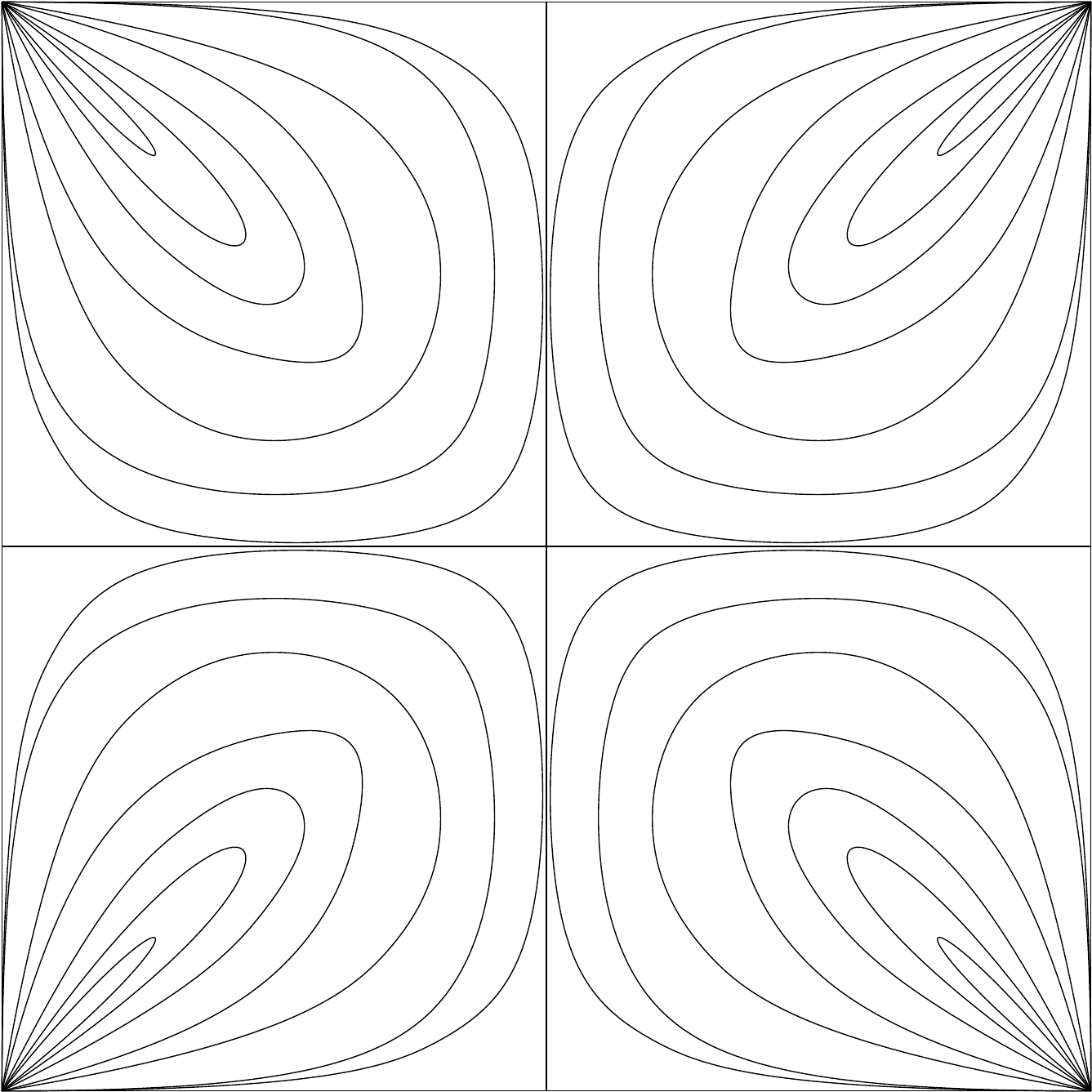}}  \caption{\label{fig:foliation-torus_crusher_a}
  Arrangement (a) horizontal and vertical foliations for $\Phi^h$ on $\Sigma _1$}
\end{figure}

\subsection{Background}
We first describe the development of the theory of harmonic maps from Riemann surfaces, in order to frame the peculiarities of our problem and also put our theorems in context.

\subsubsection{Surface harmonic mapping history}

Harmonic maps between compact Riemann surfaces of the same topological complexity have been well studied.
They are well-behaved - for example, such harmonic maps which are isotopic to the identity are in fact diffeomorphisms (\cite{SY78}, \cite{Sampson78}).

Harmonic maps between surfaces of the same genus have also found applications in Teichm\"uller theory. Starting with a harmonic diffeomorphism between Riemann surfaces, smoothly deforming the target Riemann surface structure induces a smooth variation of the harmonic map \cite{EL81}. Wolf showed that, under these deformations, the associated Hopf differential on the domain Riemann surface is uniquely determined \cite{Wolf89} and, furthermore, that the Teichm\"uller space $\mathcal{T}_g$ of a genus $g$ compact Riemann surface is homeomorphic to the vector space $QD(\Sigma _g)$ of holomorphic quadratic differentials on the fixed Riemann surface domain. Preceding this, Sampson observed that the mapping of the target hyperbolic metric to its corresponding Hopf differential is injective \cite{Sampson78}, and Hitchin constructed hyperbolic metrics for each holomorphic quadratic differential \cite{hitchin87}.

For harmonic maps between surfaces of different genera, less has been said. At a local scale, Wood gave a classification of admissible folds and cusps for harmonic maps between compact surfaces through a description of the admissible germs of a harmonic map \cite{Wood77paper}. On a global scale, some examples have been obtained. Non-holomorphic harmonic maps from positive genus surfaces to the sphere $\C P ^1$ have been produced by cutting and gluing methods \cite{Lemaire78}. A qualitative example of a harmonic map from a genus two surface to the torus $\C / \mathbb{Z}^2$ was described by integrating a pair of Abelian differentials with appropriate periods \cite{Wood74thesis}. Less is known for harmonic maps between compact negatively curved Riemann surfaces of different genera.

Somewhat tangentially, harmonic maps from $\C$ to $\HH$ have recently received a lot of attention \cite{wan92}, \cite{WA94}, \cite{HTTW95}, \cite{ST02}. Sparked by Schoen's question on the nonexistence of a harmonic diffeomorphism $\C \rightarrow \HH$ (see \cite{Schoen93}), considerable effort has been put toward studying the Bochner equation on $\C$, a semi-linear elliptic partial differential equation necessarily satisfied by the holomorphic energy density of a harmonic map from $\C$. It is remarkable that furnishing a solution to this equation satisfying some geometric constraints provides sufficient data to produce a harmonic diffeomorphismof $\C$ into an ideal polygon of $\HH$, called the \emph{Scherk maps}, through the methods of \cite{WA94}, \cite{HTTW95}, and \cite{ST02}.

This produces for us many examples of harmonic diffeomorphisms from a simply connected domain onto negatively curved surfaces. In a sense, these maps can be seen as local models for harmonic maps from a punctured sphere to a negatively curved target which are only diffemorphisms. Our task is to extend these models to produce handle collapses, harmonically mapping a punctured Riemann surface to the hyperbolic plane.

This paper establishes the existence and uniqueness of harmonic maps from once-punctured Riemann surfaces with arbitrary but finite topology to the interior of a regular ideal polygon in the hyperbolic plane. The existence of these maps can be viewed both as an extension and application of the harmonic mapping theory of $\C \rightarrow \HH$. On the one hand, we generalize the punctured sphere domain $\C$ to a punctured Riemann surface domain $\Sigma$ with positive genus. On the other hand, the construction will involve a local blow-up argument; this blow-up forces us to analyze the behavior of these $\C \rightarrow \HH$ Scherk maps.

\subsubsection{Remark on approach}

As mentioned, a common approach to exhibiting the existence of a harmonic diffeomorphism from a planar domain is through constructing solutions to the Bochner equation. It is worth noting that the Bochner equation on a punctured Riemann surface has many solutions. The development of these solutions to produce folding, non-injective harmonic maps is not straight-forward, since it is not clear where and how to induce the folding.

For a concrete example of this phenomenon and difficulty, it is enough to compare two harmonic maps whose holomorphic energy densities solve the same Bochner equation: Fix a Riemann surface domain $\Sigma _g$. Choose any harmonic map $v:\Sigma _g \rightarrow \Sigma _h$ which is non-injective, and denote its Hopf differential by $\Phi ^v$. Since $QD(\Sigma _g)\approx \mathcal{T}_g$, there exists a Riemann surface $\Sigma _g ^{'}$ also of genus $g$ and for which there exists a harmonic diffeomorphism $u:\Sigma _g \rightarrow \Sigma _g ^{'}$ with Hopf differential $\Phi ^u \equiv \Phi ^v$. Since the Hopf differentials agree, their Bochner equations are the same, so we cannot foresee the folding behavior from the information of the holomorphic energy densities (the solutions to the Bochner equation) alone.

So, we abandon the approach of solving for and developing a solution of the Bochner equation, and choose instead the approach for proving existence by compact exhaustion. The main difficulties arise in estimating the energy of the harmonic maps on compacta from above and below, along a sequence of compact exhaustion. Our energy estimates of the handle crushing harmonic map on a compactum revolves around comparing two other harmonic maps to each other - the Scherk maps and the \emph{parachute maps} (introduced in \S\ref{section:toyproblem}).

Heuristically, we describe the reasoning of our approach as follows: We remove a disk from the Scherk map (which is to say: an open disk of $\C$ and its image under the Scherk map), and deform this mapping of an annulus to the unique energy minimizer in its homotopy class with the same ideal boundary conditions (the parachute map). We show that the boundary of the removed disk does not need to move very far relative to its image under the Scherk map. This suggests that, if we replaced the disk with a compact surface with non-trivial topology, the map would be nearly harmonic.

\subsubsection{Remark on technique}
This paper presents at least four methods for dealing with proper non-injective harmonic maps from non-compact surfaces:
\begin{enumerate}
\item To obtain a lower bound on the energy of a harmonic map on a compact subset, we observe that a partially-free boundary value harmonic mapping problem solution has less energy than its completely Dirichlet boundary value harmonic mapping problem counterpart (see the Doubling Lemma \ref{lem:doublinglemma}).
\item To estimate the length of the image of the free boundary for harmonic maps solving a partially-free boundary value problem, we relate the pointwise Hopf differential norm on the free boundary to the energy density on the free boundary (see Proposition \ref{prop:qd-equals-core-energy}).
\end{enumerate}
In our proof, we need to show that a sequence of harmonic mappings from a sequence of nested compact domains to a regular ideal polygon of $\HH$ (solving a sequence of Dirichlet boundary value problems) have locally bounded Hopf differentials. We achieve this through an analysis of the Laurent series expansions for their sequence of Hopf differentials (see Proposition \ref{prop:bounded-qd-subseqn}):
\begin{enumerate}
\item[3.] We bound the infinite tails of the Laurent series (in norm) by applying the Cauchy integral formula along the prescribed Dirichlet boundaries. This uniformly bounds all but finitely many coefficients, so that divergence of the Hopf differentials could only occur by divergence of these central coefficients.
\item[4.] We control the central terms of the Laurent series using a blow-up argument on large disks in the Hopf differential norm.
\end{enumerate}

Using these methods, we are able to compare the energy of a harmonic map on a compactum of the punctured Riemann surface to the difference between energy of two different harmonic maps, the Scherk map and the parachute map, on an \emph{annular} domain. We are able to bound this difference because the energy of the parachute map can be estimated in terms of its energy density on the free boundary, which in turn we are able to bound by showing that the Laurent series for its Hopf differential is bounded.

\subsection{Organization of paper}
Throughout the paper, we will assume that $g>1$ and $k>2$ an even positive integer have been fixed.

In \S\ref{section:preliminaries}, we recall definitions and previous results on harmonic maps.

In \S\ref{section:construction}, we prove Theorem \ref{thm:crusher_existence} (Existence). This section begins with a brief outline of the proof. We will assume the Energy Estimate (Lemma \ref{lem:energyestimate}) and the results on harmonic maps recalled in \S\ref{section:preliminaries}. The Energy Estimate will be proven in \S\ref{section:toyproblem},

In \S\ref{section:uniqueness}, we prove Theorem \ref{thm:crusher_uniqueness} (Uniqueness), describing the senses in which the harmonic maps from Theorem \ref{thm:crusher_existence} are unique.

In \S\ref{section:square-example}, we specialize Theorems \ref{thm:crusher_existence} and \ref{thm:crusher_uniqueness} to the case of a punctured square torus mapped to an ideal square. We investigate the Hopf differential of this map, deducing Corollary \ref{cor:torus_crusher}.

In \S\ref{section:toyproblem},  we study a toy problem: does there exist a symmetric harmonic map from an infinite cylinder to the hyperbolic plane? After a brief outline, we pose and solve this problem, whose solutions we call \emph{parachute maps}. The section closes with a derivation of the Energy Estimate (\ref{lem:energyestimate}).


\section{Preliminaries and notation}\label{section:preliminaries}
In this section, we recall basic definitions and collect some results on harmonic maps for later use. For these results, we expound on the relevant background and also briefly mention the context of our application. We also include a glossary of the notation we will utilize throughout this discussion.

\subsection{Definitions}
Let $(M^2,\sigma)$ be a Riemann surface and $(N^n,\rho)$ be a smooth Riemannian $n$-manifold. Let $u:(M,\sigma)\rightarrow (N,\rho)$ be a $\mathcal{C}^{\infty}$ mapping. The \emph{energy} of the map $u$ is defined as
$$ E_M(u):= \int_{M} ||Du||_{TM\tensor u^{*}(TN)} ^2 \sigma (z) \, dz \, d\zbar.$$
Critical points of this energy functional are characterized by its Euler-Lagrange equation, a system of elliptic semi-linear partial differential equations which form the components of the \emph{tension field} of $u$:
$$\tau _u (z) \equiv 0.$$
Adopting a local complex coordinate $z\in M$, the tension field for $u$ reduces to
$$ \tau _u (z) = u_{z \zbar} (z) + (log \rho) _z u_z u_{\zbar}.$$
We call these critical points \emph{harmonic maps}.

The Riemann surface structure on $(M^2, \sigma)$ defines a natural complex structure $z$ on $M$. Using this complex structure, the boundary operator $d:T^* M \rightarrow T^*M$ on the deRham cohomology of $M$ decomposes as $$d = dz\wedge  \nabla_{\partial _z} + d\zbar \wedge \nabla_{\partial _{\zbar}}.$$ This allows us to decompose the vector bundles of symmetric tensors $M$ into \emph{types}. In particular, we have the decomposition $$ T^*M \odot T^* M = \langle dz \odot dz \rangle \oplus \langle dz \odot d\zbar \rangle  \oplus \langle d\zbar \odot d\zbar \rangle$$ into the types $(2,0)$, $(1,1)$, and $(0,2)$, where $\langle dz \odot dz \rangle$ denotes the $\mathcal{C}^{\infty}(M)$ span of $dz \odot dz$ in $T^*M \odot T^*M$, etc. Note that a $(2,0)$ tensor field is not necessarily holomorphic (for example, $f(\zbar)dz ^2$ is not holomorphic for any $f\neq 0$ holomorphic with respect to $z$).

The \emph{Hopf differential} $\Phi ^u$ of $u\in \mathcal{C}^{\infty}(M^2, N^2)$ is the $(2,0)$-part of the pullback metric:
$$ \Phi ^u (z) := (u^*(\rho))^{(2,0)}.$$
The holomorphic and anti-holomorphic energies, $\mathcal{H}^u$ and $\mathcal{L}^u$, of the map $u$ are given by
\begin{align*}
\mathcal{H} ^u :=||\partial u || ^2 &= \frac{\rho}{\sigma} |u_z|^2 \\
\mathcal{L}^u := ||\overline{\partial} u || ^2 &= \frac{\rho}{\sigma} |u_{\zbar}|^2
\end{align*}
Using a complex coordinate $z$ on $M$, we have the following formulae:
\begin{align*}
e^u & =\mathcal{H}^u + \mathcal{L}^u\\
\mathcal{J}^u & =\mathcal{H}^u - \mathcal{L}^u\\
||\Phi ^u ||^2 &= \mathcal{H}^u \cdot \mathcal{L}^u \cdot \sigma\\
\Delta _M &\equiv \frac{4}{\sigma} \frac{\partial ^2}{\partial z \, \partial \zbar}\\
K^M &= -\frac{1}{2} \Delta _M log (\sigma) \\
K^N &= -\frac{1}{2} \Delta _N log (\rho)
\end{align*}
These curvatures are related to the holomorphic and anti-holomorphic energies of a harmonic map. Recall that $\mathcal{J}^u = \mathcal{H}^u - \mathcal{L}^u$. We have the following \emph{Bochner formulae for harmonic maps}:
\begin{align*}
\Delta _M \log \left(\mathcal{H}^u\right) & = -K^N \mathcal{J}^u + K^M\\
\Delta _N \log \left( \mathcal{L}^u\right)&= K^N \mathcal{J}^u + K^M
\end{align*}

We remark that \emph{orientation preserving} harmonic maps $u$ can be analyzed a priori through their Bochner equations by imposing the ansatz $\mathcal{J}^u \geq 0$. Using the change of variables $w=\frac{1}{2}\log \left(\mathcal{H}^u\right)$, this yields a desirable sign on the holomorphic energy Bochner equation
$$ \Delta _{M} w = e^{2w} - ||\Phi||^2 e^{-2w}$$
so that a version of the Maximum Principle holds. This Maximum Principle has been widely applied (c.f. \cite{SY78}, \cite{Sampson78}, \cite{Wolf89}, \cite{WA94}, \cite{TD13}). Our handle crushing harmonic maps cannot be orientation preserving, so we have to develop different techniques for their analysis.

\subsection{Previous results on harmonic maps}
Here, we collect previous results on harmonic maps which we will apply in our proof of Theorem \ref{thm:crusher_existence}. 

\subsubsection{Harmonic maps $\C \rightarrow \HH$}
The Scherk maps are the basis for our investigation. We will use them to provide appropriate boundary conditions for the harmonic maps on compacta $\Sigma _s$ of the punctured Riemann surface $\Sigma$.

\begin{theorem}\label{thm:scherkmap}\emph{(Scherk maps)}
Let $\Phi$ be a polynomial holomorphic quadratic differential of degree $k$ on $\C$. If $\Phi$ is not identically zero, then there exists a unique $\mathcal{C}^{\infty}$ solution to $$\Delta W = e^{2W} - |\Phi|^2 e^{-2W}$$ for which $e^{2W}-|\Phi|^2 e^{-2W} \geq 0$ and $e^{2W} |dz|^2$ is a complete Riemannian metric on $\C$. Furthermore, these solutions are in bijection with harmonic maps $w:\C \rightarrow \HH$ for which the Hopf differential $\Phi_w$ of $w$ is identically $\Phi$. Furthermore, $w(\C)$ forms the interior of an ideal $(k+2)$-gon.
\end{theorem}
This statement is an application of theorems from \cite{WA94}, \cite{HTTW95}, and \cite{ST02}. In \cite{WA94}, the Bochner equation is solved to produce constant mean curvature cuts in Minkowski space. These are graphical, totally space-like surfaces of constant meant curvature in $\R^{2,1}$ parameterized by $W:\C \rightarrow \R$, where we identify $\C \cong \{(x,y,0)|x,y\in \R\} \subset \R^{2,1}$. Their generalized Gauss maps $w:\C\rightarrow \HH$ have image in the space-like sphere $\HH = \{ (x,y,z)|x^2 + y^2 - z^2 = -1\}$. Furthermore, $w$ is harmonic precisely when the graph of $W$ has constant mean curvature \cite{Milnor83}.

The proof of Theorem \ref{thm:scherkmap} exploits the fact that the partial differential equation solved by $W$ coincides with the Bochner equation for the holomorphic energy of a harmonic map $w$. It is remarkable that a solution to the Bochner equation is sufficient to produce a harmonic map $w:\C \rightarrow \HH$, since this $\mathcal{H}^w =\frac{1}{2}log ||\partial w||$ necessarily solves the Bochner equation when $w$ is harmonic.

In \cite{HTTW95}, the closure of the image of a Scherk map is characterized to be an ideal polygon. In fact, they also prove that any harmonic map $\C \rightarrow \HH \times \R$ with image lying in the convex hull of finitely many ideal points (i.e., an ideal polygon) must have a polynomial Hopf differential.

\subsubsection{Courant-Lebesgue Lemma} Classically, harmonic maps from a fixed surface domain into a compact target space have been shown to exist by the direct method: consider a sequence of maps with decreasing energy in a fixed homotopy class of continuous maps, and show that the limit exists. The uniform bound on the energy of the maps leads to a uniform bound on the modulus of continuity by the Courant-Lebesgue Lemma \cite{Courant37}:

\begin{theorem}\label{thm:courantlebesgue} \emph{ (Courant-Lebesgue Lemma)}
Consider a $\mathcal{C}^2$ map $u :\Sigma ^2\rightarrow N$ with energy $E(\Sigma ^2, u) \leq K$. Then, for all $p\in \Sigma ^2$ and $\delta \in (0,1)$ there exists $\rho \in (\delta, \sqrt{\delta})$ so that, for all $x, y \in \{ q\in \Sigma ^2 | d_{\Sigma ^2}(q,p)=\rho \}$, we have $$d_{N}(u(x),u(y)) \leq \sqrt{\frac{8 \pi K}{\log \left(\delta^{-1}\right)}},$$ where $d_N$ denotes the intrinsic distance on $N$.
\end{theorem}

For our non-compact setting, we will look at a sequence of harmonic maps on a compact exhaustion. In order to get them to uniformly sub-converge, it will be required to restrict them onto a common compact domain. We will apply the Courant-Lebesgue Lemma to obtain equicontinuity of these restrictions.

\subsubsection{Wood's analysis of images of harmonic maps}
The following two results are contained in \cite{Wood74thesis}, but we recall their cleaner statements from different sources. The following was adapted from the proof of Theorem 5.1 (and its corollary) of \cite{Wood77paper}, and places restrictions on the Euler characteristic of the image of a harmonic map:

\begin{theorem}\label{thm:wood_gauss-bonnet}\emph{(Gauss-Bonnet formula for images of harmonic maps)}
Let $u:(M,\sigma)\rightarrow (N,\rho)$ be a harmonic map between real analytic surfaces, where $M$ may have boundary $\partial M$. In the case $\partial M\neq \emptyset$, we only assume $\left.u\right|_{\partial M}$ is smooth on the boundary and that $\left.u\right|_{int(M)}$ is harmonic in the interior of $M$. The total curvature $\mathcal{K}(u(M))$, the geodesic curvature of the boundary image, and the Euler characteristic $\chi (u(M))$ of the image of $u$ are related by the inequality
\begin{align*}
\mathcal{K}(u(M)) + \sum_{\gamma \in \partial u(M)} \kappa_{\gamma} \geq 2\pi \chi (u(M)).
\end{align*}
In particular, if $\sum_{\gamma \in \partial M} \kappa_{\gamma} = 0$ and if the metric $\rho$ on $N$ has non-positive curvature, then $u(M)$ cannot be contractible.
\end{theorem}

This adaptation is proven exactly as the original theorem, without omitting the prescribed boundary value terms. 

\subsubsection{Dumas-Wolf's uniqueness of orientation-preserving harmonic maps through Hopf differentials}

We will consider parachute maps to study the energies of the harmonic maps from compacta $\Sigma _s$ to $\HH$. In studying the behavior of the quadratic differentials of these parachute maps, we will use a blow-up analysis of the parachute maps, leading to harmonic maps from $\C\rightarrow \HH$. We will characterize the possibilities for those blow-ups using the following result which holds in the simply connected domain setting (Theorem 5.3 of \cite{DW15}):
\begin{theorem}\label{thm:dumas-wolf_uniqueness}
For any polynomial holomorphic differential $\Phi$ of degree $k$, there is a unique complete and non-positively curved metric $\sigma (z ) |dz|^2$ on $\C$ for which the curvature satisfies $$K_{\sigma} = \left( -1 + |\Phi| _{\sigma} ^2\right)$$
\end{theorem}

This theorem leads to a characterization of harmonic maps because the pull-back metric of a harmonic map from $\C$ to $\HH$ has a metric with Gaussian curvature expressed by this formula.

\subsubsection{Cheng's Lemma}
The classical Liouville Theorem states that a bounded entire function on $\C$ must be constant. Yau extended the Liouville theorem by generalizing the domain of the harmonic function: any bounded harmonic function on a manifold of non-negative Ricci curvature must be constant \cite{Yau75}.
Cheng takes this a step further by generalizing the target manifold: if $M$ is a Riemannian manifold on non-negative Ricci curvature, and $N$ is a simply connected Riemannian manifold of non-positive sectional curvature, then any harmonic map $u:M \rightarrow N$ whose image $u(M)$ lies in a compact subset of $N$ must be constant \cite{Cheng80}.

This final generalization was obtained in \cite{Cheng80} as a consequence of a gradient estimate for harmonic maps:

\begin{theorem}\label{thm:chenglemma} \emph{(An interior gradient estimate)}
Let $M$ be an $m$-dimensional Riemannian manifold with Ricci curvature bounded below by $-K \leq 0$. Let $B_a (x_0)$ be the closed geodesic ball of radius $a$ and center $x_0$ in $M$, and let $r$ be the distance function from $x_0$. Let $f: M \rightarrow N$ be a harmonic map into $N$ which is simply connected and having non-positive sectional curvature. Let $y_0$ lie outside of $f(B_a(x_0))$ and let $\rho$ denote the distance from $y_0$. Let $b > sup \{\rho(f(x)) | x\in B_a (x_0)\}$. Then we have on $B_a(x_0)$: $$\frac{(a^2-r^2)^2|\nabla f| ^2}{(b^2 - \rho ^2 \circ f) ^2} \leq c_m max\left\{\frac{K a^4}{\beta}, \frac{a^2(1+Ka)}{\beta}, \frac{a^2b^2}{\beta} \right\} $$ where $c_m$ is a constant depending only on the dimension of $M$ and $\beta = inf \{b^2 - \rho ^2 \circ f (x)|  x\in B_a (x_0)\}$.
\end{theorem}

\subsubsection{Choe's iso-energy inequality}
The parachute maps we will introduce solve a partially-free boundary value harmonic mapping problem. In order to estimate its 1-dimensional energy on the free boundary, we can relate this energy to the energy of a harmonic map from a disk ``filling in" the free boundary image. We will be able to do this via (Theorem 2.4 of \cite{Choe98}):
\begin{theorem}\label{thm:isoenergy}\emph{(Iso-energy inequality)}
Let $u$ be harmonic map from a ball $B$ in $\R ^n$ into a non-positively curved manifold. Let $E(u, B)$ denote the energy of $u$ and $E(u|_{\partial B}, \partial B)$ denote the energy of $u|_{\partial B}$. Then $$(n-1) E(u,B) \leq E(u|_{\partial B}, \partial B) .$$
\end{theorem}

\subsubsection{Glossary of notation}
Throughout the construction, we will compare so many maps to each other that it may be useful to have a dossier for the maps, their domains, and some associated quantities. We collect below the many symbols used throughout our discussion.

\begin{align*}
& \mbox{\bf{Domains}} \\
\Sigma & := \mbox{ the punctured Riemann surface}\\
\Sigma _s & : = \mbox{ a compactum of a compact exhaustion for the punctured Riemann surface}\\
\Omega _s & : = \mbox{ the disk of radius } s\\
\Omega _{r,s} & := \mbox{ the annulus } \Omega _s \setminus \Omega _r \mbox{ (identified with } \Sigma _{r,s} := \Sigma _s \setminus \Sigma _r ) \mbox{ for } r>>1\\
\gamma _s &: = \mbox{ the boundary } \partial \Omega _s\\
\Gamma _s &: = \mbox{ the image of } \gamma _s \mbox{ under } w\\
\mathcal{P}_k & : = \mbox{ the $k$-sided regular ideal polygon in $\HH$}\\
& \\
& \mbox{\bf{Functions and associated quantities}} \\
\tau_f &:= \mbox{ the tension field of a } \mathcal{C}^2 \mbox{ map } f\\
\Phi ^f (z) & := \phi _f (z) dz^2 = \mbox{ the Hopf differential of the map } f\\
\mathcal{H}^f &= \mbox{ the holomorphic energy of the map } f\\
\mathcal{L}^f &= \mbox{ the anti-holomorphic energy of the map } f\\
e^f(z) &= \mbox{ the (pointwise) energy density of the map } f\\
E(D, f) &= \mbox{ the energy of the map } f|_{D} \mbox{ on the domain }D\\
\\
& \mbox{\bf{Auxilliary harmonic maps (in order of appearance)}} \\
w & : = \mbox{ the Scherk map } \C \rightarrow \HH \mbox{ with an ideal square image}\\
h_s & := \mbox{ a harmonic map } \Sigma _s \rightarrow \HH \mbox{ with Dirichlet boundary conditions}\\
u_s & := \mbox{ the parachute map } \Omega _{r,s} \rightarrow \HH \mbox{ solving a partially free BVP}\\
\end{align*}

\section{Construction: Proof of Theorem \ref{thm:crusher_existence}}\label{section:construction}
In this section, we prove Theorem \ref{thm:crusher_existence}, assuming the The Energy Estimate (Lemma \ref{lem:energyestimate}), along with the previous results on harmonic maps from \S\ref{section:preliminaries}. The proof will proceed by producing a sub-converging sequence of harmonic maps from compact domains exhausting a genus $g$ punctured Riemann surface to an ideal $k$-sided polygon in $\HH$. To organize ideas, we structure the proof in stages. Each sub-section is devoted to a stage of the proof:

\begin{itemize}
\item[\S\ref{section:construction-set_up}:]We will define harmonic maps $h_s$ on compacta of a punctured Riemann surface $\Sigma$.
  \begin{enumerate}
    \item Define a genus $g$ Riemann surface domain $\Sigma$ and a compact exhaustion $\cup_{s} \Sigma _s$ of $\Sigma$.
    \item Define a $k$-sided ideal polygon $\mathcal{P}_k \subset \HH$.
    \item Define harmonic maps $h_s$ from these compacta $\Sigma _s$ to $\mathcal{P}_k \subset \HH$.
  \end{enumerate}
\item[\S\ref{section:construction-compacta_subconvergence_condition}:]We will determine conditions for sub-convergence for these maps $h_s$.
  \begin{enumerate}
    \item Aiming to diagonalize, fix $r>0$ and restrict all maps $h_s$ (for $s>r$) to the domain $\Sigma _r$.
    \item Bound the energies $E(\Sigma _r , h_s)$ of these restrictions  $\left. h_s \right|_{\Sigma _r}$ in terms of the energies on \emph{annuli} $\Omega_{r,s}$ of the Scherk map and the $h_s$ away from the handles.
    \item Deduce a sufficient condition for uniform sub-convergence of the restrictions, expressed as a difference between the energies of the restrictions $\left.h_{s}\right|_{\Sigma _r}$ and the Scherk map $w$ on \emph{annuli} $\Omega_{r,s}$.
  \end{enumerate}
\item[\S\ref{section:construction-comparison_maps}:]We will produce comparison maps $u_s$ to obtain a more amenable Relaxed Sufficient Condition (Lemma \ref{lemma:relaxed-sufficient-condition}), so that we only need to bound the difference between the energies of the $u_s$ and the Scherk map $w$ on \emph{annuli} $\Omega_{r,s}$. Then, we cite the Energy Estimate to verify that the Relaxed Sufficient Condition this is satisfied by the restrictions $\left. h_s \right|_{\Sigma _r}$.
  \begin{enumerate}
    \item Define the \emph{parachute maps} $u_s$. Their analysis is deferred to \S\ref{section:toyproblem}.
    \item The energies of these $u_s$ are compared to the energies of the restrictions  $\left. h_s \right|_{\Sigma _r}$ via the Doubling Lemma (Lemma \ref{lem:doublinglemma}).
    \item Use the Energy Estimate (Lemma \ref{lem:energyestimate}) to estimate the energy of the parachute maps.
  \end{enumerate}
\item[\S\ref{section:construction-proof}:]We will provide a proof of Theorem \ref{thm:crusher_existence}.
  \begin{enumerate}
    \item Use the estimates from \S\ref{section:construction-comparison_maps} to show that the restrictions $\left. h_s \right|_{\Sigma _r}$ satisfy the Relaxed Sufficient Condition (Lemma \ref{lemma:relaxed-sufficient-condition}) for sub-convergence from \S\ref{section:construction-compacta_subconvergence_condition}.
    \item Check for non-collapsing of the limit.
    \item Analyze the Hopf differential and finish proof of \ref{thm:crusher_existence}.
  \end{enumerate}
\end{itemize}

\subsection{Set-up for construction}\label{section:construction-set_up}

\subsubsection{Notation and set-up} Let us begin by establishing some notation for various domain and range spaces involved, and collect some required known results. For $\HH$, we will use the Poincar\'e unit disk model $$\HH = \left(\{|\zeta|<1| \zeta \in \C \}, \frac{2}{\sqrt{1-|\zeta|^2}} d\zeta \right).$$

{\em Scherk map:} There is a harmonic map $w: \C \rightarrow \HH$ having Hopf differential $\Phi ^w (z) = z^{k-2} dz ^2$ whose image is the ideal $k$-gon in $\HH$ and which is symmetric with respect to the $x$-axis, in the sense that pre- and post- composition with the $x$-axis reflectional symmetries on $\C$ and $\HH$ leaves the map invariant (existence by Theorem \ref{thm:scherkmap}).

This map can be realized as the composition of the projection of the Scherk-type minimal surface in $\HH \times \R$ onto $\HH$ following the parameterization from $\C$ of the minimal surface \cite{Wolf07}. The minimal surface is conformally of type $\C$ because it has finite total curvature. The Scherk surface exists \cite{MRR11} and has a polynomial Hopf differential \cite{Wolf07}. That its Hopf differential is $-z^{k-2} dz^2$ follows by study of symmetric maps in \cite{ST02} and \cite{AW05}.

{\em Domain data:} The punctured Riemann surface serving as our domain is constructed as an identification space, and is denoted by $\Sigma = \Sigma _g$. The conformal model for $\Sigma$ used in this construction is obtained by identifying $2g +1$ subsets of $\C$ by translation:
\begin{eqnarray*}
L &=& \C \setminus \left\{ \left(\left[-\frac{1}{2},\frac{1}{2}\right] + \frac{i}{2} \right) \cup \left( \left[-\frac{1}{2},\frac{1}{2}\right] - \frac{i}{2}\right) \right\} \\
H_+^j &=& \left\{ z\in \C \suchthat -\frac{1}{2} \leq Re(z) \leq \frac{1}{2}, -\frac{i}{2} \leq Im(z) \leq \frac{i}{2} \right\} \\
H_-^j &=& \left\{ -z\in \C \suchthat -\frac{1}{2} \leq Re(z) \leq \frac{1}{2}, -\frac{i}{2} \leq Im(z) \leq \frac{i}{2} \right\}
\end{eqnarray*}
for $1\leq j \leq g$. The identifications are depicted in Figure \ref{fig:identification-surface-model}, with decorations indicating opposite sides of a geodesic segment in the identification. We will denote by $\Sigma _s$ the compact subset $$\Sigma _s := (\Omega _s \cap L)\cup \left[\mathop{\cup}_{j=1}^g \left[H_+^j \cup H_-^j\right]\right]$$ of $\Sigma$ obtained by the same identifications. Note that $$\Sigma = \mathop{\cup} _{s =1} ^ \infty \Sigma _s$$ provides a compact exhaustion of $\Sigma$.

\begin{figure}
\begin{center}
\includegraphics[scale=.65]{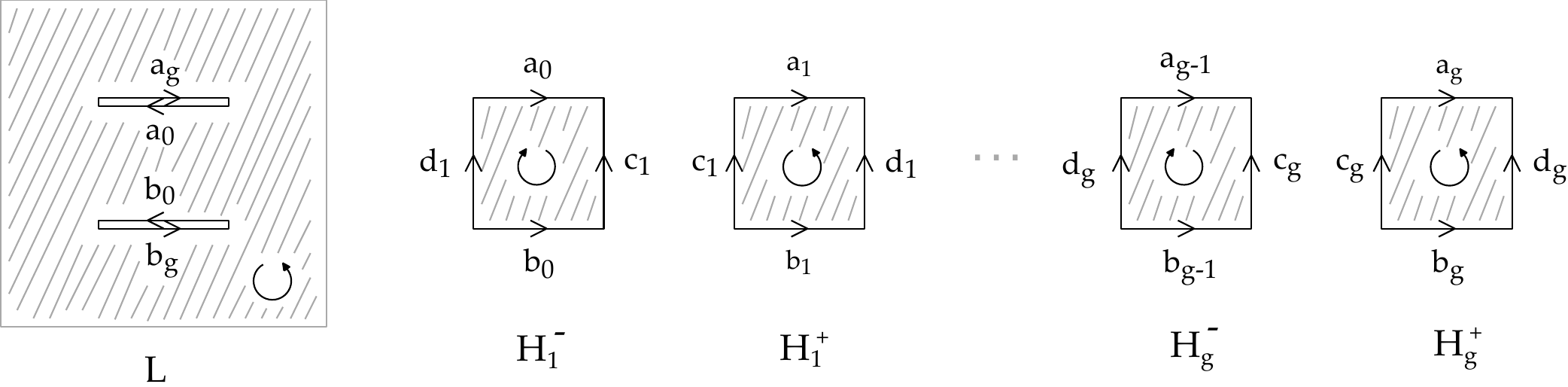}
\caption{\label{fig:identification-surface-model}Identification space model for $\Sigma _g$}
\end{center}
\end{figure}

\emph{Other domains:} Let $\Omega _r$ be the disk $\{ z \suchthat |z| \leq r \} \subset \C$, so that it has boundary $\gamma _r \:= \partial (\Omega _r)$. Let $\Gamma _s$ denote the image of $\gamma _s$ under $w$. The curves $\Gamma _s := w(\gamma _s)$ parameterized by $w|_{\gamma _s}$ will serve as boundary data for the sequence of harmonic mappings $h_s$ on compacta of the punctured Riemann surface $\Sigma$. The Scherk map is defined on $\Omega _r$, and the parachute map will be defined on the cylinder $\Omega_{r,s}$. We will also identify $\Sigma _s \setminus \Sigma _r $ by $\Omega _{r,s}$ for all $s>r$. A depiction of the Scherk map is provided in Figure \ref{fig:scherk-map}.

\begin{figure}[h!]
\begin{center}
\includegraphics[height=1.4in]{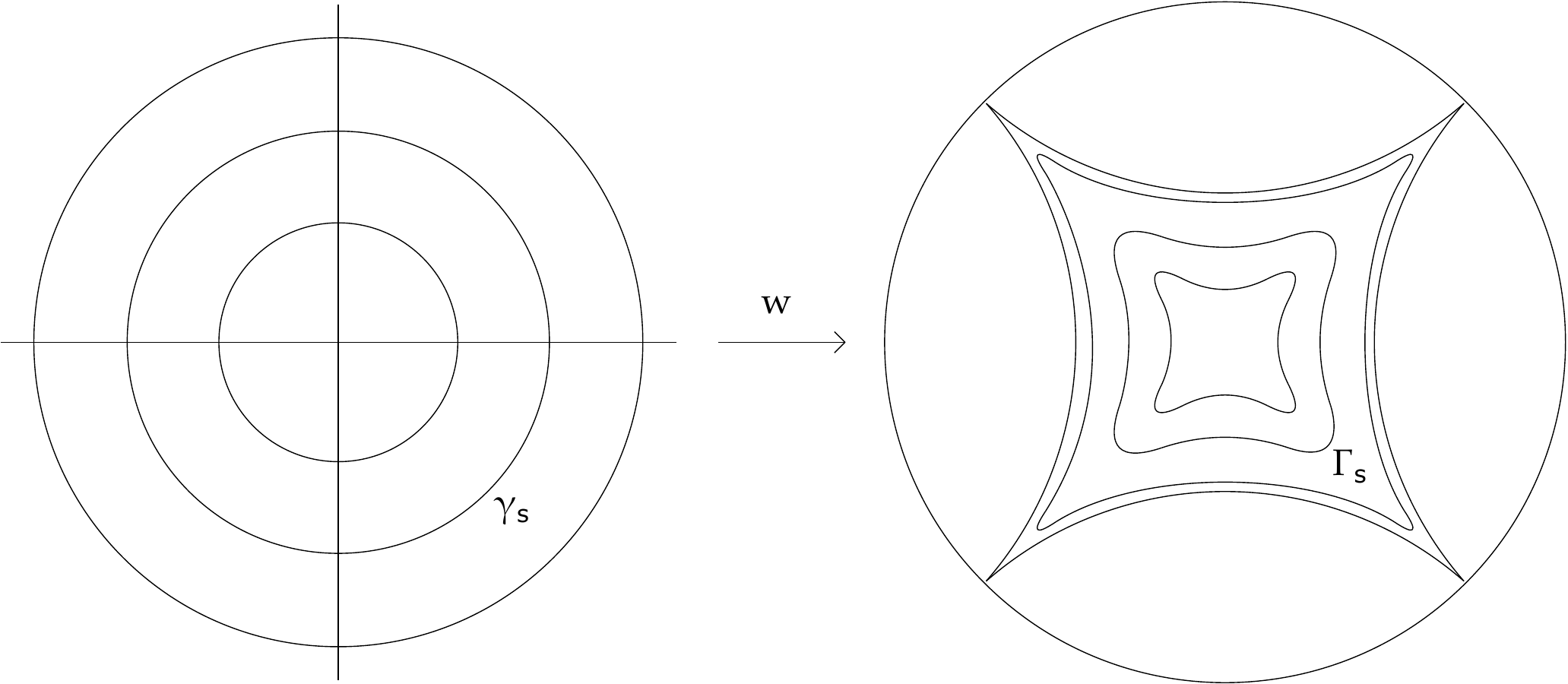}
\caption{\label{fig:scherk-map}The Scherk map $w:\C\rightarrow\HH$ with $\Phi ^w = -z^2 dz^2$.}
\end{center}
\end{figure}

{\em Target data:} The regular ideal $k$-sided polygon $\mathcal{P}_k$ we will use as the target for our harmonic maps is given by the convex hull of the the ideal points $$\{\xi _0 , \xi _2, \ldots, \xi  _{k-1} \}\subset \partial \HH$$
where we are using the Poincare disk model for $\HH$ and
$$ \xi_j := e^{2\pi i\frac{1}{2k}}\cdot e^{2\pi i \frac{j}{k}}. $$
It is fixed set-wise by the involution reflecting over the $x$-axis and also by the rotations about the origin by integer multiples of $\frac{2\pi}{k}$. Note the factor of $e^{2\pi i\frac{1}{2k}}$ to create the reflectional symmetry over the $x$-axis. The polygon $\mathcal{P}_k$ is depicted in Figure \ref{fig:ideal-polygon-examples} for various values of $k$.

\begin{figure}[h!]
  \centering
  \subfloat[$\mathcal{P}_4$]{\includegraphics[scale=.3]{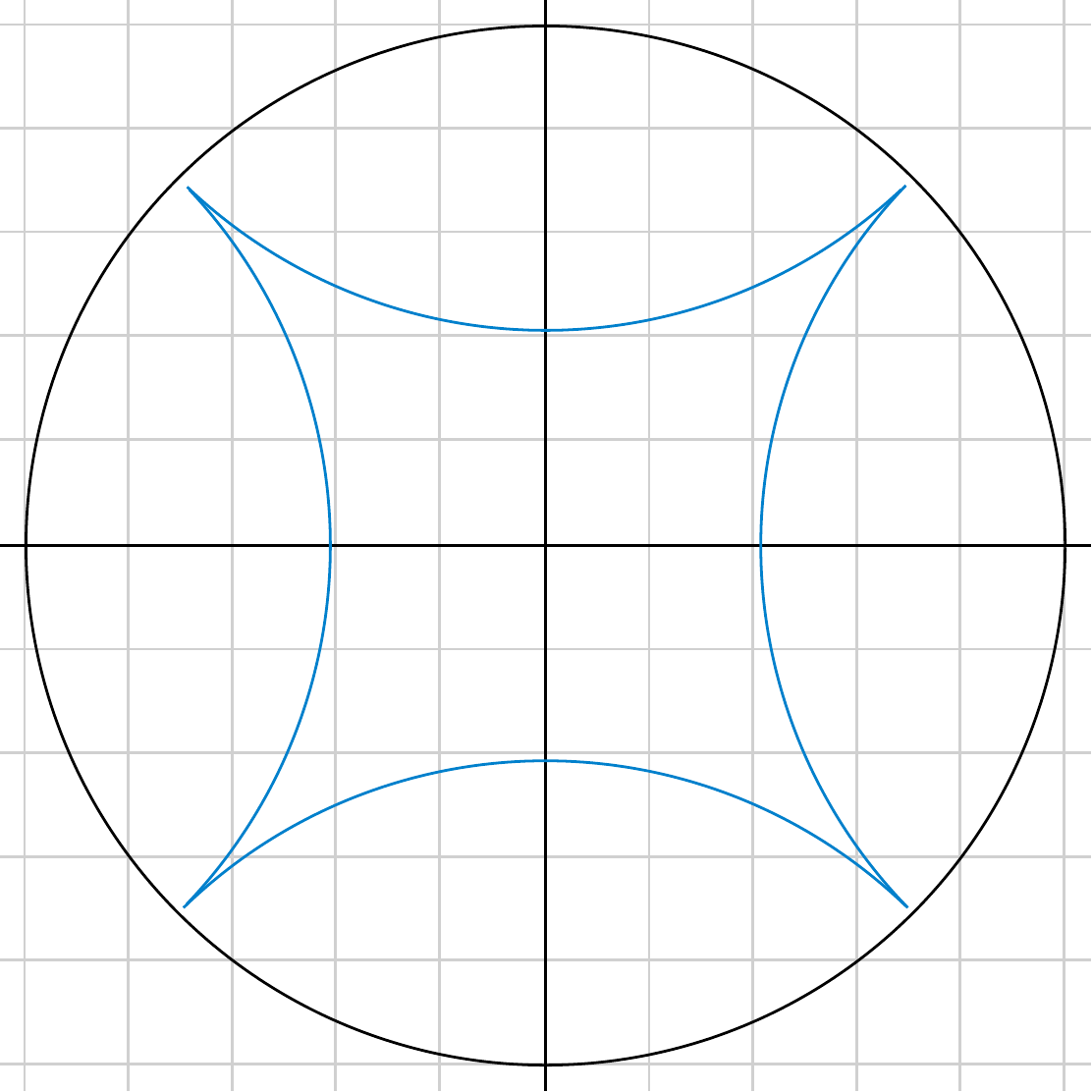}} \hskip 11pt
  \subfloat[$\mathcal{P}_6$]{\includegraphics[scale=.3]{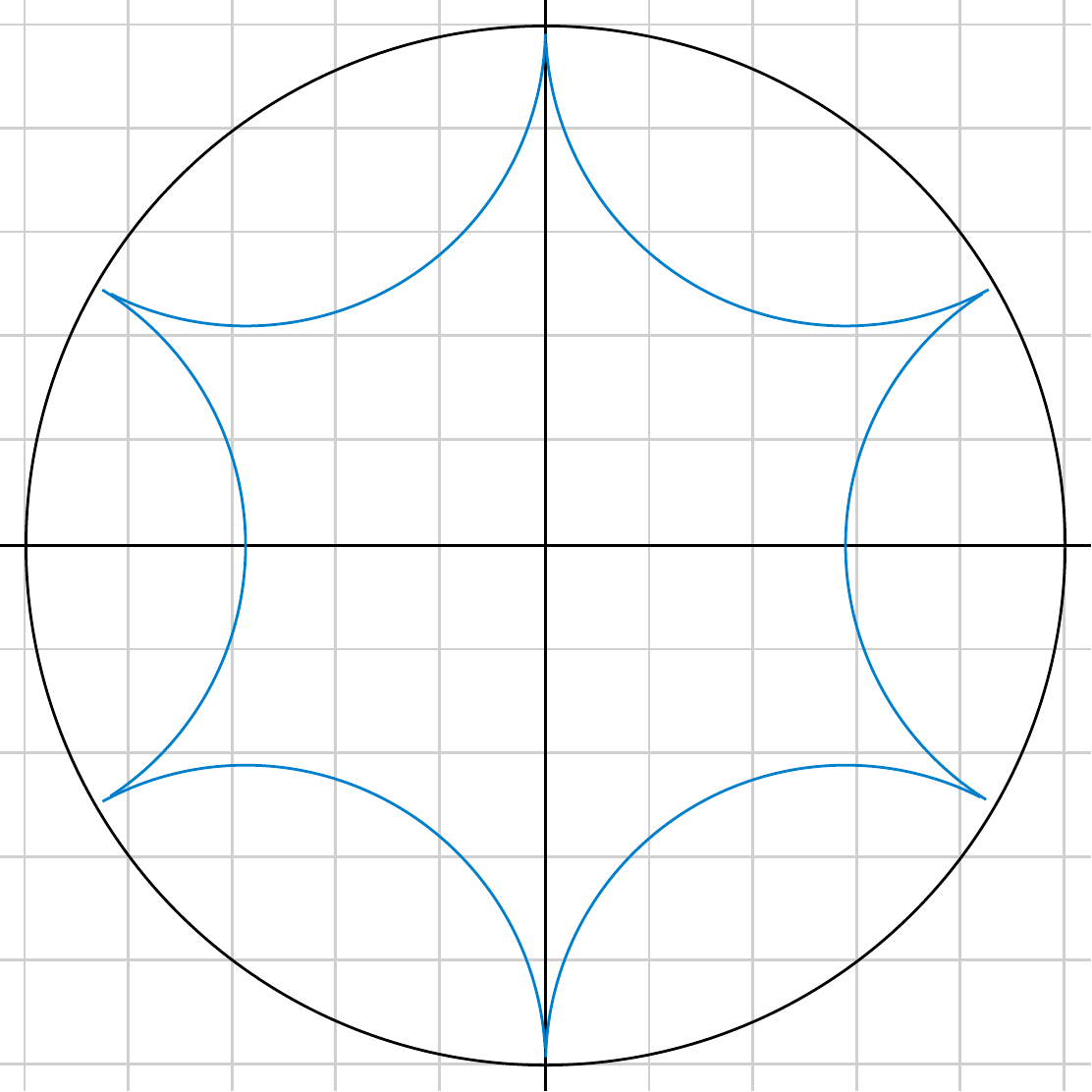}}\hskip 11pt
  \subfloat[$\mathcal{P}_8$]{\includegraphics[scale=.3]{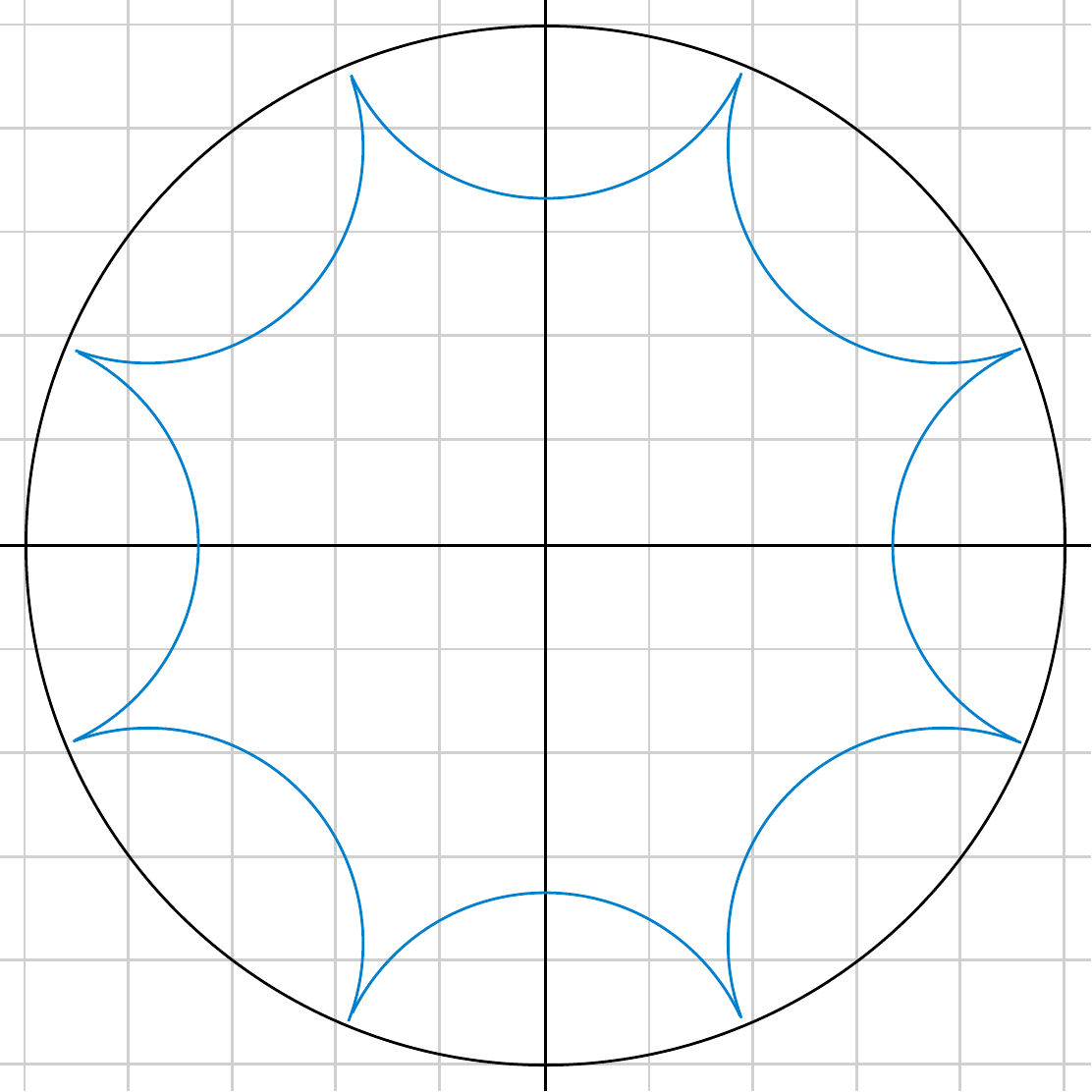}} \caption{\label{fig:ideal-polygon-examples}Example $\mathcal{P}_k$ for $k =4$, $6$, and $8$.}
\end{figure}

{\em Comparison model (the parachute maps):} Along the way, we will study harmonic maps $u_s$ from cylinders into $\HH$ which have comparable energy to the harmonic mappings $h_s$ of $\Sigma _s$ on the subsets $\Omega_{r,s} \equiv \Sigma_s \setminus \Sigma_r $ (for $r < s$). This allows us to pass from studying the (energy of the) harmonic mapping of $\Omega_s$ to harmonic mappings of $\Omega _{r,s}$ - reducing the consideration of a punctured torus domain to an analytically and topologically simpler annular domain.

\subsubsection{Sequence $\{h_s : \Omega _s \rightarrow \HH\}$ of harmonic mappings on compacta}
When $r>\frac{1}{2}$, the domain $\Omega _r$ is a neighborhood of the handles formed by identifying $\cup_{j=1}^g\left[H _+ ^j \cup H_- ^j\right]$ to $L$. So, fix $r>>\frac{1}{2}$; the precise size of $r$ will be chosen later, when it matters in estimating the geometry of the boundary values $\Gamma _r$ we will prescribe.

For each $s >r$, we pose the Dirichlet harmonic mapping problem
\begin{equation}\label{eqn:Ds}
\begin{split}
h_s :\Omega _s \rightarrow \HH \\
h_s |_{\gamma _s}=w|_{\gamma _s}\\
\end{split}\tag{Ds}
\end{equation}
noting the parameterized boundary condition. This boundary condition makes sense to impose because, although the domains for $h_s$ and $w|_{\Omega _s}$ (being $\Sigma _s$ and $\Omega _s$) are different, their charts have $\gamma _s$ as identical boundaries. There is a unique solution $h_s$ to  (\ref{eqn:Ds}) which minimizes energy in its homotopy class fixing the boundary, shown to exist by Lemaire in \cite{Lemaire82}. 

This sequence of harmonic mappings $$ h_s : \Sigma _s \rightarrow \HH$$ is the focus of our construction, from which we will extract a sub-convergent sequence. We note that this sequence of harmonic maps has an increasing amount of energy $E(\Sigma _s, h_s) \nearrow +\infty$, so that it is not immediate that there should exist a convergent subsequence.

\subsection{Conditions for sub-convergence}\label{section:construction-compacta_subconvergence_condition}
Since we wish to show that $h_s$ converges uniformly on compact subsets, we now restrict the maps $h_s$ to a common compact domain $\Sigma _r$, for a fixed $r$ and for all $s>r$. We appeal to the Arzela-Ascoli Theorem to guarantee a convergent subsequence of $\{ h_s |_{\Sigma_r} \}$ on each $\Sigma_r$, i.e., it suffices to show that $\{ h_s |_{\Sigma_r} \}$ is uniformly bounded and equicontinuous. These will both follow from an application of the Courant-Lebesgue Lemma \ref{thm:courantlebesgue} and a geometric consideration of $h_s$. We can pin down one point so that its image does not escape to infinity along the entire subsequence:

\begin{proposition}\label{prop:origin_is_real}
For all $s$, we have $d_{\HH}\left(h_s (0) , \mathcal{O}\right)\leq C(k),$ where the constant $C(k)$ is given by $$\label{constant:origin_bound} C(k):=\log\left(\frac{1 + \sec\left(\frac{\pi}{k}\right) - \tan\left(\frac{\pi}{k}\right)}{1 - \sec\left(\frac{\pi}{k}\right) + \tan\left(\frac{\pi}{k}\right)}\right)$$
and $mathcal{O}$ is the origin in $\HH$. In particular, using the Poincare disk model for $\HH$, we have that $0\in L$ maps to a point of $\{x+i\cdot 0\suchthat |x| <C(k) \}\subset \HH$.
\end{proposition}
\begin{proof}
Recall that the polygon $\mathcal{P}_k$ has the ideal points $$\{\xi _0 , \xi _2, \ldots, \xi  _{k-1} \}\subset \partial \HH$$
where
$$ \xi_j := e^{2\pi i\frac{1}{2k}}\cdot e^{2\pi i \frac{j}{k}}. $$
Indexed this way, the involution
\begin{align*}
\iota : \HH &\rightarrow \HH\\
 p \mapsto \overline{p}
\end{align*}
takes $\xi _j \mapsto \xi _{(k-j) \mod k}$. Note that $k$ is even, so there are no ideal fixed points.

Observe that the maps $h_s(z)$ and $\overline {h_s} (\zbar)$ agree pointwise for all $z\in \gamma _s=\Sigma _s$. Since there is a unique energy minimizing solution to the Dirichlet boundary value problem \eqref{eqn:Ds}, it must be that $$h_s(z)=\overline {h_s} (\zbar)$$ on the interior of $\Sigma _s$ as well. Thus, we must have
$$h_s(0)=\overline {h_s} (\overline{0}).$$

This implies that $h_s (0)$ is real. Since $h_s (\Sigma _s) \subset \mathcal{P}_k$, the image of $0\in L$ must lie in the geodesic segment
\begin{align*}
&\mathcal{P}_k \cap \{x+i\cdot 0 | x\in(-1,1)\}\\
&\equiv \left(\tan\left(\frac{\pi}{k}\right) - \sec\left(\frac{\pi}{k}\right) , \sec\left(\frac{\pi}{k}\right) - \tan\left(\frac{\pi}{k}\right)\right)\times \{ 0\} \subset \HH
\end{align*}
along the $x$-axis is fixed by the isometry $i$. That the endpoints are these follows from an elementary exercise in hyperbolic geometry. Finally, note that the length of this line segment is
$$ 2\log\left(\frac{1 + \sec\left(\frac{\pi}{k}\right) - \tan\left(\frac{\pi}{k}\right)}{1 - \sec\left(\frac{\pi}{k}\right) + \tan\left(\frac{\pi}{k}\right)}\right),$$
and that $\mathcal{O}$ is the midpoint of this line segment.
\end{proof}

We now deduce a sufficient condition for sub-convergence of our sequence of restrictions $\left.h_s\right|_{\Sigma _r}$:

\begin{proposition}\label{prop:precompactness}
If $E(\Omega _r, h_s)<M(r)$ is bounded for all $s>r$, then $\{ h_s | _{\Omega _r}\}$ is precompact.
\end{proposition}
In the following sections, the bulk of our work will be toward obtaining an energy bound of the form $E(\Omega _r, h_s) < M(r)$ for all $s>r$.

\begin{proof}
With the energy bound in hand, the Courant-Lebesgue Lemma implies that $h_s |_{\Omega_r}$ is equicontinuous and uniformly bounded, because $E(\Omega _r, h_s) < M$ gives us the inequality
$$d_{\HH}(h_s (z),h_s (w)) \leq \sqrt{\frac{8 \pi M}{\log \left(\mu^{-1}\right)}}.$$
By applying the triangle inequality and Proposition \ref{prop:origin_is_real}, the uniform bound
\begin{align*}
d_{\HH}(h_s (z),\mathcal{O}) &= d_{\HH}(h_s (z),h_s (0)) + d_{\HH}(h_s (0), \mathcal{O})\\
& \leq \sqrt{\frac{8 \pi M}{\log \left(\mu ^{-1}\right)}} + C(k)
\end{align*}
for all $z, w \in \Omega_r$, and where the constants $\mu=\mu(r)$ is defined as in Theorem \ref{thm:courantlebesgue} depending only on the diameter of $\Sigma _r$ and $C(k)$ is defined in Proposition \ref{prop:origin_is_real}. Thus, we can apply the Arzela-Ascoli theorem to the family $\{ h_s |_{\Omega _r}\}$.
\end{proof}

\subsubsection{Relating sub-convergent condition to the Scherk maps}

For now, let us first consider $h_s$ on its full domain $\Omega_s$, on which we can obtain a crude energy bound.

\begin{proposition}\label{prop:crude_upper_bound} For all $s>1$, we have the energy upper bound
$$E(\Sigma _s,h_s) \leq \ E(\Omega _s,w) + K$$ where the constant $K=K(g,k)$ depends only on $g$ and $k$ (through $w$) and is given by $$K = 2g \cdot E(H_+ ^1,w) < \infty $$
\end{proposition}
\begin{proof}
To estimate $E(\Omega _s, h_s)$, let us construct a map $v_s : \Omega _s \rightarrow \HH$ satisfying the boundary conditions for the Dirichlet harmonic mapping problem \eqref{eqn:Ds}, and whose energy we can compute. Since $v_s$ will be a candidate to the mapping problem \eqref{eqn:Ds}, and $h_s$ minimizes energy among all candidates to \eqref{eqn:Ds}, we will have
\begin{align*}
E(\Omega _s,h_s) &\leq E(\Omega _s,v_s).
\end{align*}
To define an appropriate $v_s$, recall that the conformal model for $\Sigma$ (and $\Sigma _s$ by restriction) is given by identifying the $2g +1$ pieces:
\begin{align*}
L &= \C \setminus \left\{ \left(\left[-\frac{1}{2},\frac{1}{2}\right] + \frac{i}{2} \right) \cup \left( \left[-\frac{1}{2},\frac{1}{2}\right] - \frac{i}{2}\right) \right\} \\
H_+ ^j &= \left\{ z\in \C \suchthat -\frac{1}{2} \leq Re(z) \leq \frac{1}{2}, -\frac{i}{2} \leq Im(z) \leq \frac{i}{2} \right\} \\
H_- ^j &= \left\{ -z\in \C \suchthat -\frac{1}{2} \leq Re(z) \leq \frac{1}{2}, -\frac{i}{2} \leq Im(z) \leq \frac{i}{2} \right\}
\end{align*}
for $1\leq j \leq g$. All of these can be considered as subsets of $\C$, the domain for the Scherk map $w$. So, we can define a map $v_s : \Omega _s \rightarrow \HH$ piecewise by restricting $w$ to $L$, $H_+ ^j$, and $H_- ^j$. Respecting changes in orientation, this means, for $1\leq j \leq g$:
\begin{align*}
v_s|_{L}(z) &:= w|_{L}(z) \\
v_s|_{H_+ ^j}(z) &:= w|_{H_+ ^j}(z) \\
v_s|_{H_- ^j}(z) &:= w|_{H_- ^j}(-z)
\end{align*}
To be clear, $v_s$ is defined so that $v_{s'} |_{\Omega _s} = v_s$ for all $s' > s$. Hence,
\begin{align*}
 E(\Omega _s,v_s) & = E(L\cap \Omega _s, w) + \sum _{j=1}^g \left[E(H_+ ^j,w)+E(H_- ^j,w)\right] \\
 \Rightarrow  E(\Omega _s,v_s) & \leq  E(\Omega _s,w) + K
\end{align*}
where the constant $K$ is defined by
\begin{align*}
K &:= \sum_{j=1}^g \left[ E(H_+ ^j,w)+E(H_- ^j,w) \right] \\
& = 2g \cdot E(H_+ ^1,w) < \infty
\end{align*}
is simply the energy of the Scherk map on the pieces that form the handles $\cup _{j=1}^g \left[ H_+^j \cup H_- ^j \right]$ for the domain $\Omega _s$. Thus, we have
\begin{align*}
E(\Sigma _s,h_s) & \leq E(\Sigma _s,v_s)\\
& \leq E(\Omega _s,w) + K
\end{align*}
\end{proof}

Recall that in order to obtain a convergent subsequence, we need to verify the hypothesis for Proposition \ref{prop:precompactness}: for an arbitrary fixed $r$, the restrictions $\{h_s|_{\Sigma _r}\}_{s\geq r}$ have uniformly bounded energy. However, we only have an energy estimate for $h_s$ on its entire domain of definition $\Sigma _s$ by Proposition \ref{prop:crude_upper_bound}. We do not have a priori estimates on the energies of the restricted maps.

Our first step to assuage this problem is by producing an a priori bound on $E(\Sigma _s, h_s)$ in terms of the energies of other harmonic maps which we can better understand - the Scherk map $w$ and the restriction $h_s|_{\Sigma_{r,s}}$. By decomposing the domain $\Sigma _s$ as $\Sigma _s = \Sigma _r \cup \Sigma_{r,s}$, and applying the Proposition \ref{prop:crude_upper_bound}, the energies of the restrictions can be bounded above by comparing maps of \emph{annuli}:

\begin{proposition}\label{prop:compacta-energy-scherk-bound} There exists a non-negative continuous function $G(r)$ so that, for all $s>r$, we have:
$$E(\Sigma _r,h_s) \leq E(\Omega_{r,s},w) - E(\Sigma _{r,s},h_s) + G(r).$$
\end{proposition}
\begin{proof}
The proof is entirely algebraic. Observing that $\Sigma _s = \Sigma _r \cup \Sigma_{r,s}$, we deduce:
\begin{align*}
E(\Sigma _s,h_s) &=  E(\Sigma _r,h_s) + E(\Sigma _{r,s},h_s) & \mbox{ since $\Sigma _s = \Sigma _r \cup \Sigma_{r,s}$}\\
  \Rightarrow E(\Sigma _r,h_s) &=   E(\Sigma _s,h_s) - E(\Sigma _{r,s},h_s)& \mbox{ by rearranging terms}\\
   &\leq  \left[ E(\Delta _s,w) + K \right] - E(\Sigma _{r,s},h_s) &\mbox{ by Proposition \ref{prop:crude_upper_bound}}\\
   &\leq  \left[ E(\Delta _r,w) + E(\Omega_{r,s},w) + K \right] - E(\Sigma _{r,s},h_s) & \mbox{ since $\Delta _s = \Delta _r \cup \Omega_{r,s}$} \\
  \Rightarrow E(\Sigma _r,h_s) &\leq  E(\Omega_{r,s},w) - E(\Sigma _{r,s},h_s) + G(r)&
\end{align*}
Taking $G(r):=E(\Omega_r, w) + K$, where $K$ was defined in Proposition \ref{prop:crude_upper_bound}, we have our desired result.
\end{proof}
This inequality estimates the growth of $E(\Omega _r,h_s)$ for all $s>r$. Observe that the right hand side compares the energies of the Scherk map $w$ and the Dirichlet solution $h_s$, both restriction to the common domain $\Omega _{r,s}$. Note that $\Omega _{r,s} \equiv \Sigma_s \setminus \Sigma _r = \Sigma_{r,s},$ so we will begin to use these symbols interchangeably. In the following, we work toward bounding the difference $$E(\Omega_{r,s},w) - E(\Omega _{r,s},h_s) .$$

\subsection{A comparison map for $h_s$ away from the handles}\label{section:construction-comparison_maps}

We will now construct a comparison map whose energy will bound $E(\Sigma _{r,s},h_s)$ from below. This is desirable because such a lower bound in turn bounds $E(\Omega_{r,s},w) - E(\Omega _{r,s},h_s)$ from above.

Loosely speaking, we will show that $\left.h_s\right|_{\Omega _{r,s}}$ has almost as much energy as $w|_{\Omega_{r,s}}$ (or, said differently: we are trying to show that $h_s$ does not concentrate most of its energy in $\Sigma _r$ around the handles, and that it spreads out its energy as well as $w$ does over the annulus $\Sigma _{s} \setminus \Sigma _r \equiv \Omega _{r,s}$).

\subsubsection{Parachute maps: from compact annuli $\Omega_{r,s}$ to $\HH$}
Observe that $h_s | _{\Omega_{r,s}}$ is a candidate to the partially-free Dirichlet boundary value harmonic mapping problem
\begin{equation}\label{eqn:PFDrs}
\begin{split}
u_s :\Omega _{r,s} \rightarrow \HH \\
u_s |_{\gamma _s} = w|_{\gamma _s} \\
u_s |_{\gamma _r} \mbox{ free} \\
\end{split}\tag{PFDrs}
\end{equation}
It is well known (see, for example, \cite{Hamilton75}) that critical points $u_s$ of the energy functional (which are critical with respect to variations of the map fixing the Dirichlet boundary $\gamma_s$, without restriction on the free boundary $\gamma_r$) satisfy $$(\nabla u_s)|_{\gamma_r}(\nu)=0,$$ where $\nu$ is the outward pointing normal of $\Omega _{r,s}$ along $\gamma _r$. A priori, it is not evident that $h_s$ satisfies this condition on $\gamma _r$.

In any case, we call the unique solution of our toy problem \eqref{eqn:PFDrs} a \emph{parachute map} $u_s$. Note that it is an energy minimizer among all candidates to the harmonic mapping problem \eqref{eqn:PFDrs}, so that the energy of $h_s | _{\Omega_{r,s}}$ is no greater than that of $u_s$.

Here, we note that, but uniqueness of the solution to \eqref{eqn:PFDrs}, the parachute maps respects the same $D_k$ dihedral symmetries that the boundary curve $\Gamma _s$ respects. In particular, we have the identities
\begin{align*}
\xi_{-j} \cdot u_s(z) & u_s( \xi_{j} z)\\
\overline{u_s(z) } & = u_s(\overline{z})\\
\end{align*}
for all $z\in \Omega_{r,s}$ and for all $\xi_j = e^{2\pi i \frac{j}{k}}$, for $j = 0, 1, \ldots , k-1$.

\subsubsection{Comparison models estimate $E_{\Omega _{r,s}}(h_s)$}

The crux of the argument is a doubling lemma which allows us to recast the partially-free/Dirichlet boundary value problem as a Dirichlet/Dirichlet boundary value problem on a doubled domain and appropriately reproduced boundary conditions (where the boundary conditions correspond to the inner/outer boundary components).

\begin{lemma}\label{lem:doublinglemma}
\emph{(Doubling Lemma)} Any critical point of (\ref{eqn:PFDrs}) can be reflected to define a critical point of a Dirichlet harmonic mapping problem (defined below) that only has one critical point (which also minimizes energy).
\end{lemma}

Hence, producing a critical point of (\ref{eqn:PFDrs}), actually produces an energy minimizer among candidates to (\ref{eqn:PFDrs}). We will apply this lemma after its proof.

\begin{proof}
This is proven by a series of observations.

(0) The domain $\Omega _{r,s}$ is conformally equivalent to $\Omega_{1,\frac{s}{r}}$ under the scale $\frac{1}{r}$. So, the rescaled mapping problem on the domain $\Omega_{1,\frac{s}{r}}$ is equivalent to (\ref{eqn:PFDrs}) by parameterizing the boundary appropriately under the rescaling, i.e., $f|_{\gamma _{s/r}}(r\cdot z) = w|_{\gamma _s}(z)$.

(1) The critical points $f: \Omega_{1,\frac{s}{r}} \rightarrow \HH$ of the equivalent (\ref{eqn:PFDrs}) mapping problem
\begin{align*}
f : \Omega_{1,\frac{s}{r}} & \rightarrow \HH \\
f|_{\gamma _{s/r}} &= w|_{r\cdot \gamma _{s/r}} \\
f|_{\gamma _1} & \mbox{ free}
\end{align*}
are harmonic and satisfy $(\nabla f)|_{\gamma_1}(\nu)=0$, where $\nu$ denotes the outward normal along $\gamma _1$. The curve $\gamma_1$ appears as the free boundary of $\Omega_{1,\frac{s}{r}}$ and so, by the regularity of $f$ on $\gamma _1$, $f$ can be reflected across $\gamma_1$ onto $\Omega_{\frac{r}{s},1}$ to produce a map $$\tilde{f} : \Omega_{\frac{r}{s},\frac{s}{r}}\rightarrow \HH$$ using a Schwarz Reflection Principle.

(2) The map $\tilde{f}$ can be considered as a candidate for the Dirichlet harmonic mapping problem \begin{equation*}\label{eqn:Drs}
\begin{split}
&\tilde{f} :  \Omega_{\frac{r}{s},\frac{s}{r}} \rightarrow \HH \\
&\tilde{f}|_{\gamma_{r/s}} (z) = w|_{r\cdot \gamma _{s/r}} \left(\frac{1}{\zbar}\right)\\
&\tilde{f}|_{\gamma _{s/r}} (z) = w|_{r\cdot \gamma _{s/r}} (z).\\
\end{split}\tag{Drs}
\end{equation*}
Observe that the orientation is reversed on the boundary $\gamma_{r/s}$ to preserve the dihedral symmetry (in the target) of the mapping problem. It is automatically satisfied by definition of $\tilde{f}$ by reflection of $f$ across $\gamma _1$.

(3) Observe that (\ref{eqn:Drs}) has a unique solution $\tilde{F}$ which, in particular, minimizes energy among all maps satisfying the Dirichlet boundary condition. Furthermore, it is the unique critical point. So, it must be that $\tilde{f}=\tilde{F}$.

(4) Thus, there is a unique critical point $F$ of the (\ref{eqn:PFDrs}) mapping problem. This critical point $F$ minimizes energy, and can be realized as a rescaling by $r$ of the restriction $\tilde{F}|_{\Omega_{1,\frac{s}{r}}}$ since the domains are conformal, i.e., $$\Omega _{r,s} = r\cdot \Omega_{1,\frac{s}{r}}.$$ In other words, the solution to (\ref{eqn:PFDrs}) is given by
\begin{equation*}
u_s(z)=\tilde{F}|_{\Omega_{1,\frac{s}{r}}} (r\cdot z).
\end{equation*}
since $(\nabla u_s)|_{\gamma_1}(\nu)=0$. This completes the proof of the Doubling Lemma.
\end{proof}

Let us return to studying the energies of the restrictions $\left. h_s \right|_{\Sigma _r}$. Denote by $u_s$ the solution to (\ref{eqn:Drs}). The Doubling Lemma \ref{lem:doublinglemma} implies
\begin{equation*}
E_{\Omega_{r,s}}(u_s) \leq E_{\Omega_{r,s}}(h_s).
\end{equation*}
Recall from Proposition \ref{prop:compacta-energy-scherk-bound} that
\begin{eqnarray*}
E(\Sigma _r,h_s) &\leq E(\Omega_{r,s},w) - E(\Omega _{r,s},h_s) + G(r).
\end{eqnarray*}
Combining these, we finally arrive at
\begin{lemma}\label{lemma:relaxed-sufficient-condition}\emph{(Relaxed Sufficient Condition)}
For each $s >r$, and the for the maps $h_s$, parachute map $u^s$, and Scherk map $w$ defined above, we have
\begin{eqnarray*}
E(\Sigma _r,h_s) &\leq E(\Omega_{r,s},w) - E(\Omega _{r,s},u_s) + G(r),
\end{eqnarray*}
where $G(r)$ is the function defined in \ref{prop:compacta-energy-scherk-bound}.
\end{lemma}

This inequality allows us to finally control the energies of the restrictions solely in terms of a comparison between the energies of the Scherk map $w$ and the parachute maps $u_s$. We will now use the Energy Estimate to bound $E(\Omega_{r,s},w) - E(\Omega_{r,s},u_s)$:

\newtheorem*{lem:energyestimate}{Lemma \ref{lem:energyestimate}}
\begin{lem:energyestimate}
\emph{(The Energy Estimate)}
There is a continuous increasing function $F(r)$ such that, when $r$ is large enough, for any $s>r$, the following inequality holds:
$$E(\Omega_{r,s},w) - E(\Omega_{r,s},u_s) < F(r)$$
\end{lem:energyestimate}

We defer its proof to \S\ref{section:toyproblem-the_energy_estimate}.

\subsection{Putting it all together}\label{section:construction-proof}
With this in hand, we are finally able to diagonalize $\{ h_s\}_{s\geq r}$ on the domains $\Omega _r$ to obtain a convergent subsequence.
\begin{proof}\label{proof:crusher_existence}
Fix $r >1$ large enough to satisfy the hypotheses of the Energy Estimate (Lemma \ref{lem:energyestimate}). For each $s>r$, the energy of the map $\{ \left.h_s\right|_{\Sigma _r}\}_{s\geq r}$ can be bounded by
\begin{align*}
E(\Sigma _r,h_s) &\leq E(\Omega_{r,s},w) - E(\Omega _{r,s},u_s) + G(r) &\mbox{ by Proposition \ref{lemma:relaxed-sufficient-condition}}\\
&\leq F(r) + G(r) & \mbox{ by the Energy Estimate (Lemma \ref{lem:energyestimate})}
\end{align*}
where $F(r)$ is defined in the Energy Estimate (Lemma \ref{lem:energyestimate}) and $G(r)$ is defined in Proposition \ref{lemma:relaxed-sufficient-condition}. Thus, by Proposition \ref{prop:precompactness}, the family $\{ \left.h_s\right|_{\Sigma _r}\}_{s\geq r}$ is precompact.
\end{proof}

\subsubsection{Non-collapsing of the limit $h_s \rightarrow h$}
With the uniform energy upper bound $E(\Omega_r, h_s)<M(r)$, we can take a limit of the harmonic maps using a diagonalization process. However, we face the possibility of $h_s$ collapsing a subset of the domain $\Omega$. The classification of local singularities for maps between surfaces by Wood is collected in his paper \cite{Wood77paper}, although the derivations are more thorough in his thesis \cite{Wood74thesis}.

We have to rule out the degenerate possibility that the limiting map has an image which does not contain any open disk, i.e., that $h$ is a piecewise constant map (with a $0-$dimensional image) or maps to a union of geodesic segments. It suffices to exhibit an energy growth rate of $E(\Sigma_s, h_s)$ bound from below.

Recall that the comparison models $u_s$ satisfy:
\begin{align*}
E(\Omega_{r,s},u_s) &\leq E(\Omega_{r,s},h_s) & \mbox{ by Lemma \ref{lem:doublinglemma}}\\
E(\Omega_{r,s},w) - E(\Omega_{r,s},u_s) &\leq F(r) & \mbox{ by Lemma \ref{lem:energyestimate}}
\end{align*}
Combining these, we see that there is a lower bound on $E(\Omega_{r,s},h_s)$:
\begin{align*}
E(\Omega_{r,s},w) - F(r) & \leq E(\Omega_{r,s},h_s)
\end{align*}
since $E(\Omega_{r,s},w)$ is of order $k$ polynomial growth in $s$. This, in turn, bounds $E(\Sigma_s,h_s)$ from below since $E(\Omega_{r,s},h_s) \leq E(\Sigma _s ,h_s)$. Finally, since the maps $h_s$ converge uniformly on compact subsets to $h$, the energy densities $E(\Sigma _s ,h_s)$ converge to $E(\Omega _{r,s},h)$, and we deduce that $E(\Sigma _s ,h_s)$ is bounded from below as a function of $s$.

\subsubsection{Hopf differential $\Phi^h$ and energy $e^h$ of $h$}
Now, we turn to studying the limiting Hopf differential $\Phi ^h$. Recall that
\begin{align*}
\mathcal{H}^h&:= ||\partial h||_{\sigma,\rho}^2= |h_z|^2 \frac{\rho (h(z))}{\sigma(z)} \\
\mathcal{L}^h&:= ||\overline{\partial} h||_{\sigma,\rho}^2 = |h_{\overline{z}}|^2 \frac{\rho (h(z))}{\sigma(z)}
\end{align*}
respectively denote the holomorphic and anti-holomorphic energies of $h$. Observe that we have a pointwise inequality which helps us bound the norm of $||\Phi ^h||$:
\begin{align*}
||\Phi ^h||^2 &= \mathcal{H}^h \mathcal{L}^h \\
&\leq ( \mathcal{H}^h + \mathcal{L}^h ) ^2 \\
& \leq (e^h)^2
\end{align*}

So, we have at most polynomial growth of $$ \int _{\Sigma_r} ||\Phi ^h || dvol $$ of order $k$ as a function of $r$. With the lower bound on energy growth rate from the non-collapsing discussion above, we see that in fact the polynomial growth rate is $k$ (otherwise, the map must limit to strictly fewer than $k$ ideal geodesics). This implies that $\Phi ^h$ has a pole of order $k+2$.

Let us apply the Riemann-Roch theorem to deduce the number of zeros of $\Phi ^h$. The degree of the square of the canonical line bundle is equal to the number of zeros minus the number of poles of a non-vanishing meromorphic quadratic differential. Thus,
$$ \# \{zeros\} - \# \{ poles\} = 4g -4$$
where $g$ is the genus of the compact surface. Hence, there are $(4g -4) + (k+2)$ zeros of $\Phi _h$ on $\Sigma$. This proves the final assertion of Theorem \ref{thm:crusher_existence}, describing the Hopf differential $\Phi ^h$ of $h: \Sigma _g \rightarrow \mathcal{P}_k$.

\section{Uniqueness: Proof of Theorem \ref{thm:crusher_uniqueness}}\label{section:uniqueness}
Let us now turn to discussing the existence of the maps established in Theorem \ref{thm:crusher_existence}. We prove:

\newtheorem*{thm:crusher_uniqueness}{Theorem \ref{thm:crusher_uniqueness}}
\begin{thm:crusher_uniqueness}
Suppose $h,v: \Sigma _g \rightarrow \HH$ are two harmonic maps. Denote their pointwise distance function by $d(z) := dist_{\HH} \left(h(z),v(z)\right)$. If we have $$\left(cosh\circ d\right) -1 \in \mathcal{L} ^p (\Sigma _g)$$ for some some $p\in (1, +\infty]$, then $d(z) \equiv 0$, i.e., the maps $v$ and $h$ must agree pointwise.
\end{thm:crusher_uniqueness}

\begin{proof} Suppose we have two harmonic maps $h,v: \Sigma \rightarrow \HH$ and denote their pointwise distance function by $d(z):=dist_{\HH}(h(z),v(z))$. Our goal is to show that $d \equiv 0.$ Observe that it is equivalent to show that the function $$Q(z):=cosh(d(z))-1$$ vanishes identically on $\Sigma$. Since the target manifold $\HH$ is non-positively curved  and simply connected, $Q$ is smooth as a function of $z$. We will now recall a computation from \cite{HW97} to obtain a sign on the Laplacian of $Q$, using the Riemannian chain rule.

Recall that the Riemannian chain rule states, for a map
$$ \Delta _M (f \circ g) (z)= tr _{M} \left[\left.Hess (f) \right|_{g(z)} \circ \left(\left.Dg\right|_{z} , \left.Dg\right|_{z}\right) \right]+ \left.\nabla _N f\right|_{g(z)} \left(\left.\tau ^g\right|_{z}\right)$$

Let us choose a convenient basis for the computation. Given two points $p$ and $q$ in $\HH$, there is a unique geodesic segment $\gamma_{pq}$ between them. There are natural frames associated to the tangent spaces $T_p \HH$ and $T_q \HH$ associated to $\gamma _{pq}$: $T_p \HH = span(Tan_p , N^p)$ and $T_q \HH = span(Tan_q, N^q)$, where $Tan_p$ and $Tan _q$ are unit tangent vectors to $\gamma _{pq}$ at the points $p$ and $q$, respectively, and where $N^p$ and $N^q$ are unit normal vectors orthogonal to $Tan _p$ and $Tan _q$, respectively.

Then, using the local orthonormal frames $\{ e_1, e_2\}$ for $T \Sigma$ and $\{(Tan _p, 0), (0, N^p), (Tan_q), (0, N^q)\}$ for $T_{(p,q)}\HH \times \HH$, we have:
$$Hess_{\HH \times \HH} cosh\circ d = (cosh \circ d) {\bf I} _{4\times 4} + {\bf A} $$
where ${\bf A}$ is a $4\times 4$ matrix whose only non-zero entries are $-1$ at $((N^p,0),(0,N^p))$ and $((N^q,0),(0,N^q))$,  and a $- cosh \circ d$ at $((Tan _p,0),(0,Tan_q))$ and $((Tan _q,0),(0,Tan_p))$.

Thus, the Laplacian of $Q$ can be bounded:
\begin{align*}
\Delta Q(z) \geq & (cosh \circ d(z)) \left[ \left(\langle \nabla h, Tan_{h(z)}\rangle - \langle \nabla v, Tan_{v(z)} \rangle\right) ^2 \right]\\
& + (cosh \circ d(z)-1)\left[ \langle \nabla h,N^{h(z)} \rangle ^2 + \langle \nabla v, N^{v(z)}\rangle ^2\right]\\
& + \langle grad Q , (\tau ^h , \tau ^v) \rangle _{(h,v)(z)}
\end{align*}
where $\langle \nabla u , Tan ^{u(z)} \rangle $ denotes the projection of $du(e_1) + d(e_2)$ onto $Tan^{u(z)}$.

Since $h$ and $v$ are harmonic, the vector field $(\tau ^h , \tau ^v)$ vanishes. This implies that the function $Q(z)$ is subharmonic on $\Sigma$. Now, notice that $\Sigma _g$ is non-compact and yet we have
\begin{align*}
\Delta Q (z)& \geq 0 \\
Q(z) &\geq 0\\
Q(z)& \in \mathcal{L}^p(\Sigma )
\end{align*}
Hence, we investigate the conditions under which $d$ is forced to be a constant.

It is a classical fact from complex analysis that all bounded, non-negative sub-harmonic functions on a parabolic domain must be constant. In particular, $\Sigma _g$ is a parabolic domain because its conformal model for the puncture can be described by $\C \setminus \{0\}$. So, if $p=+\infty$, then $Q$ must be constant. Since $Q$ is also in $\mathcal{L}^p$, it must be that $Q$ vanishes identically.

For $p\in (1, +\infty )$, it was shown in \cite{Yau75} that any non-negative sub-harmonic function on a complete Riemannian manifold must be constant. It similarly follows that, in this case, $Q$ vanishes identically.

\end{proof}

\section{Example: punctured square torus squashed onto ideal square}\label{section:square-example}
In the following, we will consider the special case of the square torus in which we can determine the location of the zeros of the Hopf differential, proving Corollary \ref{cor:torus_crusher}. After this section, the remainder of the paper is devoted to the details of the energy estimate Lemma \ref{lem:energyestimate}.

\begin{proof}[Proof of Corollary \ref{cor:torus_crusher}]
Let $g=1$ and $k=4$ in the hypotheses of Theorem \ref{thm:crusher_existence}, and let $h:\Sigma_1 \rightarrow \mathcal{P}_4$ denote the harmonic handle crushing map of the square punctured torus to the ideal square. By the $D_4$ dihedral symmetry of $\Sigma _1$, each of the maps $h_s$ in the construction of $h$ is symmetric with respect to the same $D_4$ symmetries. So, the limiting map $h$ also respects the $D_4$ symmetry group. This is depicted in Figure \ref{fig:square-symmetries}; note the domain identification space $\Sigma _1$.

\begin{figure}[h!]
\begin{center}
\includegraphics[height=1.4in]{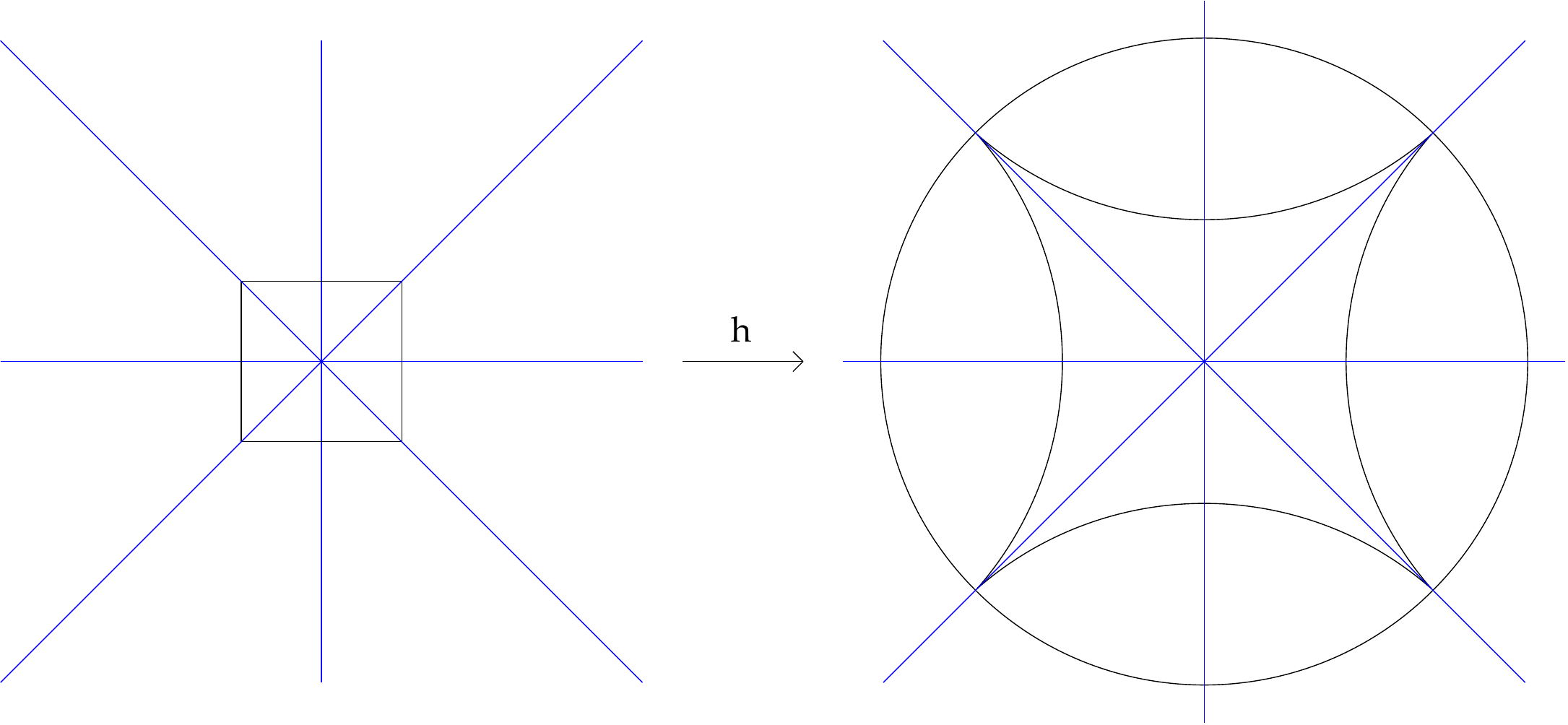}
\caption{\label{fig:square-symmetries}Lines of symmetries for $h:\Sigma _1 \rightarrow \mathcal{P}_4$ from Corollary \ref{cor:torus_crusher}.}
\end{center}
\end{figure}

By the Riemann-Roch theorem, we establish that there are $6$ zeros, counting multiplicity, of $\Phi ^h$ on $\Sigma _1$. Observe that any point $p$ not on a line of symmetry of $\Sigma _1$ cannot be a zero of $\Phi ^h$ because the $D_4$ symmetry of the map would yield $7$ other zeros of $\Phi ^h$ (the orbit of $p$ under the $8$ symmetries of $D_4$ would yield $7$ other zeros). Hence, the zeros of $\Phi ^h$ must lie on the lines of symmetry.

\begin{figure}[h!]
\begin{center}
\includegraphics[height=1.4in]{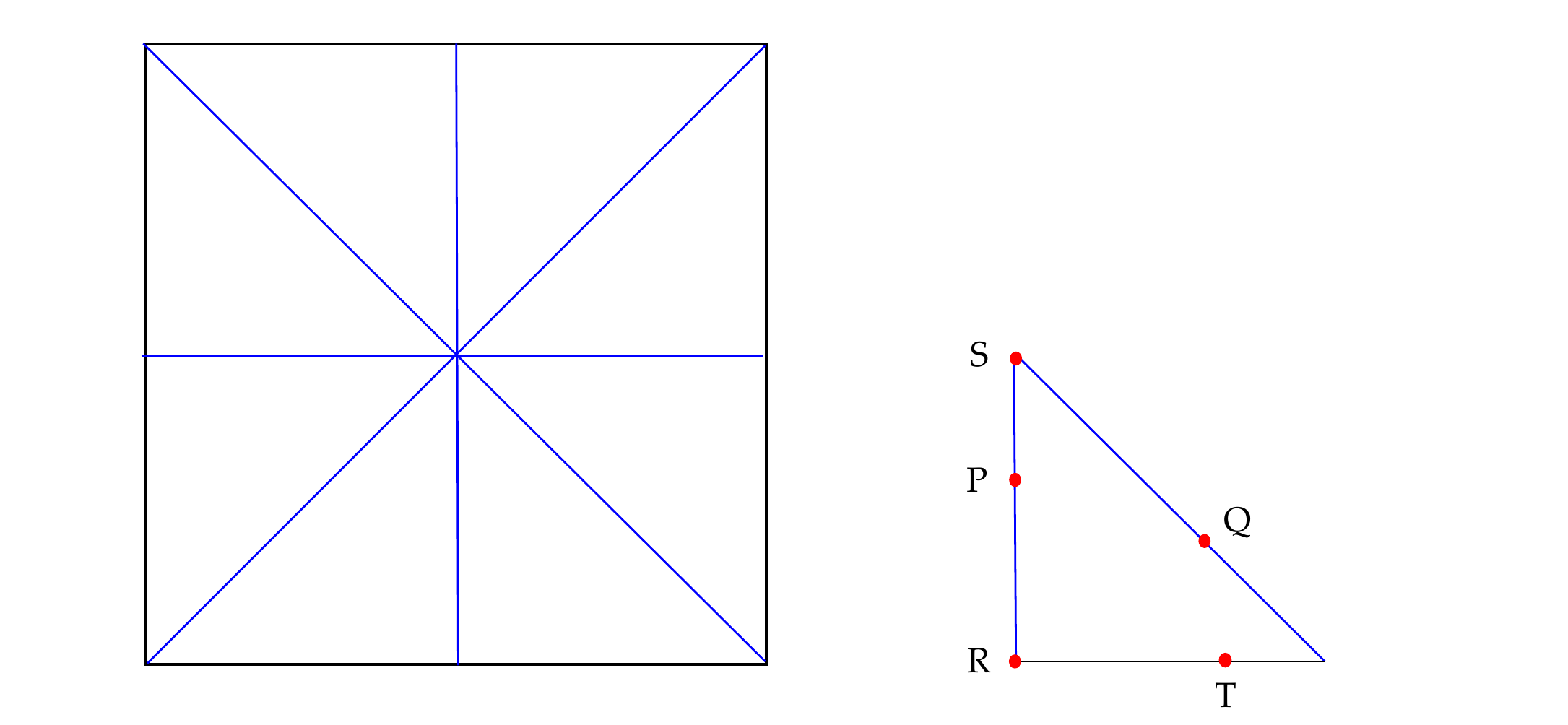}
\caption{\label{fig:square-hopf-qd-analysis}Possible locations for zeros of $\Phi ^h$ of $h:\Sigma _1 \rightarrow \mathcal{P}_4$ handle crusher.}
\end{center}
\end{figure}

Thus, there are five possible locations for a zero of $\Phi ^h$. We draw the possibilities in Figure \ref{fig:square-hopf-qd-analysis}, in a triangular fundamental domain of a conformal model of $\Sigma _1$ with respect to the $D_4$ symmetries. We label the possible locations $P$, $Q$, $R$, $S$, and $T$. Note that each of the edges of the triangular fundamental domain lies on a line of symmetry of $\Sigma _1$.

If a zero $P$ of order $n_P$ lies on a vertical line of symmetry, but not on the diagonal line of symmetry, and in the interior of the vertical line of symmetry (depicted on the vertical edge of the triangular fundamental domain), there would be $3$ other zeros of order $n_P$ by the diagonal symmetries of $D_4$.

If a zero $Q$ of order $n_Q$ lies on a diagonal line of symmetry, but not on the vertical and horizontal lines of symmetry, then there would be $3$ other zeros of order $n_Q$ by the horizontal and vertical symmetries of $D_4$.

If a zero $R$ of order $n_R$ lies on the vertical or horizontal line of symmetry, but not on the diagonal line of symmetry, and on the endpoint of the vertical or horizontal line of symmetry (depicted at the right angle of the triangular fundamental domain), there would be $1$ other zero of order $n_R$ by the diagonal symmetries of $D_4$.

If a zero $S$ is at the fixed point of the $D_4$ symmetries (depicted at the top vertex of the triangular fundamental domain). Such a zero of order $n_S$ does not necessitate other zeros of $\Phi ^h$ by the $D_4$ symmetry, for example, if $n_S = 6$. However, the foliation must respect the $D_4$ symmetry, so $n_S -2$ must be a multiple of $4$.

Finally, if a zero $T$ of order $n_T$ lies on the horizontal line of symmetry, but not on the diagonal line of symmetry, and in the interior of the horizontal line of symmetry (depicted on the horizontal edge of the triangular fundamental domain), there would be $3$ other zeros of order $n_T$ by the diagonal symmetries of $D_4$.

Hence, we must have $4n_P + 4n_Q + 2n_R + n_S +4n_T = 6$. Note that $n_P$, $n_Q$, $n_R$, $n_S$, and $n_T$ must be non-negative integers. Observe that the order $n_S$ must be such that $n_S -2$ is divisble by $4$, since the foliation of $\Phi ^h$ near $S$ must be fixed by the $D _4$ symmetries. Hence, $n_S$ is either $2$ or $6$. So, $n_P$, $n_Q$, and $n_T$ are at most $1$. Similarly, $n_R$ is at most $2$. By a case by case analysis, the only possibilities are given by the following five arrangements:

\begin{enumerate}
\item[(a)] three zeros of order $2$, with $(n_P,n_Q,n_R,n_S, n_T)=(0,0,2,2,0)$
\item[(b1)] or one zero of order $2$ and four of order $1$, with $(n_P,n_Q,n_R,n_S,n_T)=(0,0,2,0,1)$,
\item[(b2)] or one zero of order $2$ and four of order $1$, with $(n_P,n_Q,n_R,n_S,n_T)=(1,0,2,0,0)$,
\item[(b3)] or one zero of order $2$ and four of order $1$, with $(n_P,n_Q,n_R,n_S,n_T)=(0,1,2,0,0)$,
\item[(c)] or one zero of order $6$, with $(n_P,n_Q,n_R,n_S)=(0,0,0,6,0)$.
\end{enumerate}
This establishes the three possible configurations of multiplicities of zeros.
\vskip 11pt

\emph{If $\Phi ^h$ has divisor data in arrangement (a).} Let us study the orientation of the map $h$ at the zeros of $\Phi ^h$, supposing that $\Phi ^h$ has the divisor given by arrangement (a). In arrangement (a), the Hopf differential $\Phi^h$ has a zero of order two at each of the midpoints $0$ in $L$, $H_+$, and $H_-$. Let us analyze the orientation of the map at each of these zeros. Consider the curves
\begin{align*}
\eta _h &:= \{(s,0) |-1\leq s\leq 1\} \cup \{(-s,0)|-1\leq s\leq 1\} \subset H_+ \cup H_- \\
\eta _v &:= \{(0,t) |-1\leq t\leq 1\} \cup \{(0,-t)|-1\leq t\leq 1\} \subset H_- \cup L
\end{align*}
These curves generate the homology of $\Omega$. Furthermore, by the symmetries of the map, we have
\begin{align*}
h(\eta _h) \subset \{(x,0) \} \subset \HH \\
h(\eta _v) \subset \{(0,y) \} \subset \HH
\end{align*}
Since the diagonal isometry of $\Sigma _1$ taking $0\in L$ to $0 \in H_+$ and the diagonal isometry on $\HH$ are both orientation-reversing, and since pre- and post-composition of $h$ by these maps preserves $h$, the orientation of $h$ at both of these points are the same. By this, we mean $\mathcal{J} \geq0$ at these points. Now, consider the Jacobian along the curve $\eta _h$.

We argue that the sign at $0\in H_-$ is opposite to the sign at $0\in H_+$ (and $0\in L$). Consider mapping half of $\eta _h$, which connects $0\in H_+$ to $0\in H_-$ along the right side. It gets mapped into a curve along the positive real axis of $\HH$ which is a loop at $0\in \HH$. Tracing the derivative of this curve, we see that the tangent vector to the curve must change sign upon returning to $0\in L$. Hence, $\mathcal{J}^h \leq0$ here. Hence, the orientation of $h$ at two of its zeros (each of order two) are the same, while the orientation of $h$ at its other zero (of order two) is reversed.

\vskip 11pt

\emph{Only arrangement (a) is a square of a holomorphic 1-form.} It is clear that none of the configurations of the Hopf differential $\Phi ^h$ descibed by arrangement (b) can be the square of a holomorphic $1$-form $\eta$, since the simple zeros of $\Phi ^h$ cannot be realized by $\eta ^2$. If $\Phi ^h$ is in configuration (c), and if $\eta$ is a holomorphic $1$-form such that $\eta ^2 = \Phi ^h$, then $\eta$ has a zero of order $3$ at the fixed point of the $D_4$ symmetry group. By the Abel-Jacobi theorem, such a holomorphic $1$-form $\eta$ does not exist.
\end{proof}

\section{The toy problem and The Energy Estimate}\label{section:toyproblem}

This section is devoted to studying the solution of a toy problem. This solution, the parachute map, was introduced in \S\ref{section:construction-comparison_maps} to establish a priori bounds on the energy of the harmonic mapping of a compactum of the punctured Riemann surface. The energy bound hinges on the geometric quality that the free boundary of the parachute maps $u_s$ stays in a bounded set independent of $s$.

In \S\ref{section:toyproblem-parachute_no_folds} and \S\ref{section:toyproblem-parachute_qd_analysis}, we describe the geometry of the parachute map and also derive restrictions on the Laurent expansion of the Hopf differential. These will be applied in conjunction in \S\ref{section:toyproblem-bounded_core_lemma}, which is the heart of this section. We provide an outline of this section:
\begin{itemize}
\item[\S\ref{section:toyproblem-parachute_map_setup}:]Define the parachute maps.
  \begin{enumerate}
    \item Pose and solve the toy problem whose solutions are the parachute maps.
    \item Describe symmetries of parachute maps.
    \item Observe that the boundary condition at the core curve is free.
  \end{enumerate}
\item[\S\ref{section:toyproblem-parachute_no_folds}:]Describe the injectivity of the maps.
  \begin{enumerate}
    \item Use Theorem \ref{thm:wood_gauss-bonnet} (Wood's Gauss-Bonnet formula for the image of harmonic maps) to show that the parachute maps are local diffeomorphisms.
    \item Note that this lets us find disks of increasing radius in the domain along the sequence of parachute maps, from which the parachute maps are diffeomorphisms.
  \end{enumerate}
\item[\S\ref{section:toyproblem-parachute_qd_analysis}:]Analyze Hopf differentials of the parachute maps through their symmetries.
  \begin{enumerate}
    \item Express the Hopf differential as a Laurent expansion.
    \item Show that the coefficients are symmetric with respect to the $-2$ index.
    \item Observe that $\ell ^2 (\C)$-norm of the tails (indices outside $[-k-4,k]$) is finite on the core curve.
    \item Deduce that the index $-k-4$, $-2$, and $k$ coefficients determine boundedness of $|\phi ^s|$ on the core curve.
    \item Obtain a uniform bound of the Hopf differential on the core curve (the free boundary) by considering those three coefficients.
  \end{enumerate}
\item[\S\ref{section:toyproblem-bounded_core_lemma}:]Show that the core curve mapped by any parachute map stays in a bounded set.
  \begin{enumerate}
    \item Observe that the Hopf differential norm equals the energy density on the core curve.
    \item Show that the sequence of Hopf differentials have bounded norm on the core curve.
    \item[] \emph{This involves an argument by contradiction: we use a blow up argument on a sequence of disks of increasing radii in the Hopf differential norm, provided by \S\ref{section:toyproblem-parachute_no_folds}.}
    \item Use the uniform bound on Hopf differentials on the core curve to obtain the Bounded Core Lemma.
  \end{enumerate}
\item[\S\ref{section:toyproblem-the_energy_estimate}:]Use this pointwise bound on the free boundary to derive the Energy Estimate (Lemma \ref{lem:energyestimate}).
  \begin{enumerate}
    \item Use the Bounded Core Lemma to deduce that the pointwise distance between the parachute map and the Scherk map is uniformly bounded (compared on any common annular domain).
    \item Obtain a gradient norm estimate for the parachute map along the core curve by Cheng's interior gradient estimate, which relies on the uniform bound between the images of the parachute map and the Scherk map to choose appropriate constants.
    \item Observe that the parachute map energy density on the core curve reduces to simply the tangential energy of the core curve, since the core curve is a free boundary.
    \item Bound the energy of an auxilliary harmonic map from a disk which ``fills in" the core curve using Choe's Iso-energy inequality for harmonic maps (relating interior $2$-dimensional energy to $1$-dimensional energy of the boundary).
    \item Observe that the Scherk map on $\Omega _s$ has less energy than a ``filled in" harmonic extension of $\left.u_s\right|_{\Omega _{r,s}}$ into $\Omega _r$.
    \item Re-arrange the Scherk map energy bound to deduce the Energy Estimate (Lemma \ref{lem:energyestimate}).
  \end{enumerate}
\end{itemize}

\subsection{Parachute maps}\label{section:toyproblem-parachute_map_setup}
Recall the partially free boundary value harmonic mapping problem
\begin{equation}\label{eqn:PFDrs}
\begin{split}
u_s :\Omega _{r,s} \rightarrow \HH \\
u_s |_{\gamma _s} = w|_{\gamma _s} \\
u_s |_{\gamma _r} \mbox{ free} \\
\end{split}\tag{PFDrs}
\end{equation}
whose solutions $u_s$ we call parachute maps. These parachute maps were introduced in \S\ref{section:construction-comparison_maps} as comparison maps for $\left.h_s\right|_{\Omega_{r,s}}$, having the same image on the $\gamma_s$ boundary component. In this section, we analyze the discrepancy of energy between the parachute maps and the Scherk map, $E(u_s,\Omega_{r,s})-E(w,\Omega_{r,s})$, as a function of $s$, for all $s$ greater than any given fixed $r>1$. Without loss of generality, we rescale $\Omega_{r,s}$ by a factor of $\frac{1}{r}$ to avoid the factor of $\frac{s}{r}$ in the following analysis.

The parachute maps $u_s$ are closely tied to the Scherk map $w$ because the Dirichlet boundary condition on $\gamma _s$ for $u_s$ is obtained from $w|_{\gamma _s}$. Note that the energy of $u_s$ is less than that of $w$ on their common domain of definition, since $u_s$ solves the partially free boundary value problem. We are interested in the magnitude of this difference.

One can imagine the sequence of restrictions $w|_{\Omega_{1,s}}$ of the Scherk map converging (identically) to the Scherk map as $s\nearrow \infty$. On the other hand, the sequence of parachute maps $u_s$ (defined on $\Omega_{1,s}$) agree with $w|_{\Omega_{1,s}}$ on the $\gamma _s$ boundary component, but disagree on the $\gamma _1$ boundary component. For this reason, we chose the moniker \emph{parachute}, since we are reminded of the children's toy parachute getting stretched along its outer boundary twoards infinity, and the inner boundary moves freely to minimize the overall stretching of the fabric.

In comparing the energy of the parachute map to the Scherk map, then, we have a heuristic geometric description: this suggests their energies should be comparable from a geometric description: if $u_s(\gamma _1)$ remains a bounded distance from $w(\gamma _1)$ for all $s$, then the difference in energy between the parachute maps and the Scherk map (on their common domains of definition) should be small; if $u_s(\gamma _1)$ approaches $u_s(\gamma _s)$ as $s\nearrow \infty$, then the difference in energy will be great.

Our analysis involves studying the Hopf differential $\Phi ^{u_s}$ of $u_s$. However, it is prudent to study the Hopf differential of the doubled parachute maps, since an extra symmetry imposes a convenient analytic condition (see Proposition \ref{prop:parachute-qd-identity}). Recall that the Doubling Lemma \ref{lem:doublinglemma} realizes the parachute maps $u_s$ as restrictions of the solutions of a Dirichlet-Dirichlet boundary value harmonic mapping problem
\begin{align*}
u_s : \Omega_{\frac{1}{s},s} & \rightarrow \HH \\
u_s|_{\gamma _{s}} (z) &= w|_{\gamma _{s}} (z) \\
u_s|_{\gamma _{1/s}} (z)&= w|_{\gamma _{s}} \left(\frac{1}{\zbar}\right) .
\end{align*}
We use the names $u_s$ to mean either the parachute map or its double without ambiguity, since we will always specify the required domain of definition. We abbreviate $\Phi^s\equiv \Phi ^{u_s}$. The curve $\gamma _1$ will be referred to as the \emph{core curve}.

\subsection{Geometry of images}\label{section:toyproblem-parachute_no_folds}
In this section, we study the behavior of the parachute maps.

\begin{proposition}\label{prop:parachute-maps_non-negative-J}
The parachute maps $u_s$ have a non-negative Jacobian on $int\left(\Omega_{1,s}\right)$. 
\end{proposition}

\begin{proof}
Observe that the cylinder $\Omega_{1,s}$ has Euler characteristic $0$. Let us denote the subsets on which $u_s$ has positive, zero, and negative functional Jacobian by $\Omega _{1,s} ^+$, $\Omega _{1,s} ^0$, and $\Omega _{1,s} ^-$, respectively. Note that $u_s$ is orientation-preserving along the boundary $\gamma _s$, so $\Omega _{1,s} ^+$ is not empty. Suppose for a contradiction that $\Omega _{1,s} ^-$ is non-empty.

Since $\Omega _{1,s} ^0$ is the zero set of a real analytic function, it is a collection of points and curves. Hence, $ \Omega_{1,s}^+ $ and $\Omega_{1,s}^- $ are multiply connected planar domains. In particular, this implies that
\begin{align*}
\chi \left(u_s\left(\Omega_{1,s} ^+\right) \right) = \chi \left(u_s\left(\Omega_{1,s} ^-\right) \right) \in \mathbb{Z}_{\geq 0}
\end{align*}

Observe that the parachute map $u_s$ restricts onto $\Omega_{1,s}^- $ with a functional Jacobian with a definite sign. So, it is a local diffeomorphism from $\Omega_{1,s}^- $ onto its image $u_s\left(\Omega_{1,s}^- \right)$. This implies that $u_s  \left(\Omega_{1,s} ^-\right)$ is also a multiply connected planar domain. Note also that $\partial\left[ u_s  \left(\Omega_{1,s} ^-\right)\right]$ is a collection of curves for which the geodesic curvature vector points outside of the region, by the maximum principle for harmonic maps \cite{JK79}; this fact uses that we are considering $\Omega_{1,s}^-$ and not $\Omega_{1,s}^+$.

Thus, by Theorem \ref{thm:wood_gauss-bonnet}, we must have
\begin{align*}
\mathcal{K}\left(u_s\left(\Omega_{1,s} ^-\right) \right)\geq 2\pi \chi \left(u_s\left(\Omega_{1,s} ^-\right) \right) \geq 0 .
\end{align*}
We arrive at a contradiction, since $u_s\left(\Omega_{1,s} ^-\right) $ has negative Gauss curvature. Therefore $\Omega_{1,s} ^-$ must be empty, and $\mathcal{J}^u \geq 0$ on $\Omega _{1,s}$.
\end{proof}

\subsection{Hopf differentials of parachute maps}\label{section:toyproblem-parachute_qd_analysis}

In this section, we analyze the pointwise norm of the Hopf differential of the parachute maps along the core curve. The aim of the analysis is to obtain the following:

\begin{proposition}\label{prop:bounded-qd-subseqn}
For the parachute maps $u_s$ described above, there exists a positive constant $M < \infty$ and a sequence of indices $\{s \} \nearrow \infty$ so that, for all $z\in \gamma _1$, we have  $$|\phi ^s (z)|<M.$$
\end{proposition}

We will study the Hopf differential through its Laurent expansion. So, it will be useful to observe a symmetry of $\Phi ^s$ induced by the reflectional symmetries:

\begin{proposition}\label{prop:parachute-qd-identity}
Each parachute map $u_s$ respects the two reflectional symmetries $u_s \left(\frac{1}{\overline{z}}\right)=u_s (z)$ and $u_s (z)=\overline{u_s}(\zbar)$. These symmetries induce a symmetry on the Hopf differential $\Phi ^s (z)\equiv \phi ^s(z) dz^2$: $$ \phi^s(z) = \phi ^s\left(\frac{1}{z}\right) \frac{1}{z^4}.$$
\end{proposition}
\begin{proof}
Fix any $s$. We will drop the index $s$ in this proof since it is not essential, and only clutters notation. First, note that $\phi (z) = \overline{u} _z (z) \cdot u_z (z)$. Thus, it is equivalent to show that $$ u_z(z) \cdot \overline{u}_z (z) = u_z\left(\frac{1}{z}\right) \cdot \overline{u}_z\left(\frac{1}{z}\right) \cdot \frac{1}{z^4}$$

We will proceed in polar coordinates. First, note that the symmetries re-expressed as:
\begin{eqnarray*}
u(z) = u\left(\frac{1}{\zbar}\right)& \Rightarrow & u\left(r e^{i\theta} \right) = u\left(\frac{1}{r}e^{i\theta}\right)\\
u(z) = \overline{u}(\zbar) &\Rightarrow & u\left(r e^{i\theta} \right) = \overline{u}\left(r e^{-i\theta}\right)
\end{eqnarray*}
Combining these, we obtain one polar identity
$$u\left(\frac{1}{r}e^{i\theta}\right)= \overline{u}\left(r e^{-i\theta}\right)$$
Then, taking $r$ and $\theta$ derivatives, we deduce the polar derivative identities:
\begin{eqnarray}
-\frac{1}{r^2} e^{-i\theta} u_r \left(\frac{1}{r}e^{i\theta}\right) & = e^{-i\theta} \overline{u}_r \left(r e^{-i\theta}\right) \label{identity:parachute-polar_r} \\
i\frac{1}{r} u_\theta \left(\frac{1}{r}e^{i\theta}\right) &= - ir\overline{u}_\theta \left(r e^{-i\theta}\right) \label{identity:parachute-polar_theta}
\end{eqnarray}
Note that the operator $\frac{\partial}{\partial z}$ in polar coordinates becomes
\begin{align*}
\frac{\partial}{\partial z} & = e^{-i\theta} \left( \frac{1}{2} \frac{\partial}{\partial r} - i \frac{1}{r}\frac{\partial}{\partial \theta} \right)
\end{align*}
We will use these identities to verify the equality.

First, we compute $u_z (z) \cdot \overline{u}_z(z)$ to evaluate the left-hand side:
\begin{align*}
\left. \left[ u_z \overline{u}_z \right] \right|_{z} &= \left. \left[ e^{-i\theta} \left( \frac{1}{2} u_r - i r u_{\theta}  \right) \right]\right|_{z=re^{i\theta}}\cdot \left. \left[ e^{-i\theta} \left( \frac{1}{2} \overline{u}_r - i r \overline{u}_{\theta}  \right) \right]\right|_{z=re^{i\theta}}\\
& = e^{-i2\theta} \left[ \frac{1}{2} u_r (re^{i\theta}) - i r u_{\theta} (re^{i\theta}) \right] \left[ \frac{1}{2} \overline{u}_r (re^{i\theta}) - i r \overline{u}_{\theta} (re^{i\theta}) \right].
\end{align*}

Second, we compute $u_{z}\left(\frac{1}{z}\right)$, $\overline{u}_{z} \left(\frac{1}{z}\right)$, and $\frac{1}{z^4}$ and combine them to evaluate the right-hand side:
\begin{align*}
\left. \frac{\partial}{\partial z} u \right|_{\frac{1}{z}} & = \left. e^{-i\theta} \left( \frac{1}{2} u_r - i r u_{\theta}  \right)\right|_{\frac{1}{z}=\frac{1}{r}e^{-i\theta}}\\
& = e^{i \theta}\left( \frac{1}{2} u_r \left( \frac{1}{r} e^{-i\theta} \right) - i \frac{1}{r} u_\theta\left(\frac{1}{r}e^{-i\theta}\right)\right)
\end{align*}
and also
\begin{align*}
\left. \frac{\partial}{\partial z} \overline{u} \right|_{\frac{1}{z}} & = \left. e^{-i\theta} \left( \frac{1}{2} \overline{u}_r - i r \overline{u}_{\theta}  \right)\right|_{\frac{1}{z}=\frac{1}{r}e^{-i\theta}}\\
& = e^{i \theta}\left( \frac{1}{2} \overline{u}_r \left( \frac{1}{r} e^{-i\theta} \right) - i \frac{1}{r} \overline{u}_\theta\left(\frac{1}{r}e^{-i\theta}\right)\right)
\end{align*}
and of course
\begin{align*}
\frac{1}{z^4}&= \frac{1}{r^4} e^{-i4\theta}
\end{align*}
Thus, their product forms the right-hand side:
\begin{align*}
\left.\left[ u_z \cdot \overline{u}_z \right] \right|_{\frac{1}{z}} \frac{1}{z^4}&=\frac{1}{r^4} \cdot e^{-i 2\theta}\left[ \frac{1}{2} u_r \left( \frac{1}{r} e^{-i\theta} \right) - i \frac{1}{r} u_\theta\left(\frac{1}{r}e^{-i\theta}\right)\right] \cdot \left[ \frac{1}{2} \overline{u}_r \left( \frac{1}{r} e^{-i\theta} \right) - i \frac{1}{r} \overline{u}_\theta\left(\frac{1}{r}e^{-i\theta}\right)\right]\\
&= e^{-i 2\theta}\left[ \frac{1}{2} \frac{-1}{r^2} u_r \left( \frac{1}{r} e^{-i\theta} \right) + i \frac{1}{r^3} u_\theta\left(\frac{1}{r}e^{-i\theta}\right)\right] \cdot \left[ \frac{1}{2} \frac{-1}{r^2} \overline{u}_r \left( \frac{1}{r} e^{-i\theta} \right) + i \frac{1}{r^3} \overline{u}_\theta\left(\frac{1}{r}e^{-i\theta}\right)\right]
\end{align*}

Finally, we will verify that $$u_z(z) \cdot \overline{u}_z (z) = u_z\left(\frac{1}{z}\right) \cdot \overline{u}_z\left(\frac{1}{z}\right) \cdot \frac{1}{z^4}$$
by using the polar derivative identities. Note that the identities (\ref{identity:parachute-polar_r}) and (\ref{identity:parachute-polar_theta}) hold for all $r$ and $\theta$, so in particular we can compare term by term:
\begin{align*}
\frac{1}{2} u_r\left(r e^{i\theta} \right) & = \frac{1}{2} \frac{-1}{r^2} \overline{u}_r \left(\frac{1}{r} e^{-i\theta}\right) \mbox{ using $(\frac{1}{r}, \theta)$ for $(r, \theta)$ in equation \eqref{identity:parachute-polar_r}}\\
-i r u_\theta \left( r e^{i\theta}\right) &= i \frac{1}{r^3} \overline{u}_\theta \left(\frac{1}{r}e^{i\theta}\right) \mbox{ using $(\frac{1}{r}, \theta)$ for $(r, \theta)$ in equation \eqref{identity:parachute-polar_theta}}\\
\frac{1}{2}\overline{u}_r\left(r e^{i\theta}\right) & = \frac{1}{2} \frac{-1}{r^2} u_r \left( \frac{1}{r}e^{-i\theta} \right) \mbox{ using $(r, \theta)$ for $(r, -\theta)$ in equation \eqref{identity:parachute-polar_r}}\\
-i r \overline{u}_\theta \left( r e^{i\theta} \right) & = i\frac{1}{r^3} u_\theta \left( \frac{1}{r} e^{-i\theta}\right) \mbox{ using $(r, \theta)$ for $(r, -\theta)$ in equation \eqref{identity:parachute-polar_theta}}
\end{align*}
Thus, we see that: the first factor of the left hand side coincides with the second factor of the right hand side, and the second factor of the left hand side coincides with the first factor of the right hand side (with the equal summands of each factor appearing in the same order).
\end{proof}

Now we will prove Proposition \ref{prop:bounded-qd-subseqn}. We recall its statement here:

\newtheorem*{prop:bounded-qd-subseqn}{Proposition \ref{prop:bounded-qd-subseqn}}
\begin{prop:bounded-qd-subseqn}
For the parachute maps $u_s$ described above, there exists a positive constant $M < \infty$ and a sequence of indices $\{s \} \nearrow \infty$ so that, for all $z\in \gamma _1$, we have  $$|\phi ^s (z)|<M.$$
\end{prop:bounded-qd-subseqn}

\begin{proof}
Consider the Laurent series for the coefficients of the sequence of Hopf differentials:
$$ \phi^s (z) = \sum _{n=-\infty} ^{\infty} c_n ^s z^n$$
where each coefficient is obtained by Cauchy's integral formula along a contour $\gamma $
$$ c_n^s \equiv \oint _{\gamma} \frac{\phi ^s (\zeta)}{\zeta ^{n+1}} d|\zeta| $$
This expression is valid in the domain of holomorphicity of $\phi ^s$, which contains $\Omega _{\frac{1}{s},s}$

Let us get a better handle on the coefficients of these series. The $\frac{2\pi }{k}$ rotational symmetries of the map $u_s$ imply that $ \Phi (z) = \Phi\left(e^{\frac{2\pi i}{k}l}z\right)$ for each $l = 0, \ldots ,k-1$. Hence, $$ \phi(z) dz^2 = \phi\left(e^{\frac{2\pi i}{k}l}z\right) \left(e^{\frac{2\pi i}{k}l}\right)^2 dz^2$$
In terms of the Laurent coefficients, by uniqueness of the Laurent expansion for $\phi ^s(z)$, we must have
\begin{align*}
c_n^s &= c_n ^s e^{n\frac{2\pi i}{k}l} e^{2\frac{2\pi i}{k+2}l}\\
& = c_n ^s e^{\frac{2\pi i}{k+2} (n+2)l}
\end{align*}
for all $l = 0, \ldots , k-1$. In particular, for all indices $n$ for which $k$ does not divide $n+2$, the coefficient $c_n ^s$ must vanish. This leaves us with an expansion given by
\begin{align*}
\phi ^s (z) &= \cdots + \frac{c_{-(2k+2)}}{z^{{2k+2}}} + \frac{c_{-(k+2)}}{z^{k+2}} + \frac{c_{-2}}{z^{2}} + c_{k-2} ^s  z^{k-2} + c_{2k-2} z^{2k-2}+ \cdots\\
& = \left(\sum _{n=2} ^{\infty} \frac{c_{-nk-2}}{z^{nk+2}} \right)+ \left[ \frac{c_{-(k+2)}}{z^{k+2}} + \frac{c_{-2}}{z^{2}} + c_{k-2} ^s  z^{k-2} \right] + \left( \sum _{n=2} ^{\infty} c_{nk-2} z^{nk-2} \right)
\end{align*}
We separate the series into three parts as indicated by the parenthetical grouping: respectively, the negative tail, the central terms being the sum of the index $-(k+2)$, $-2$, and $k-2$ terms, and the positive tail.

Our plan of attack will be to bound the norm of each of these parts, with the bounds holding for every point on the core curve and along a sequence of indices $\{s_j\}$. Recalling the Cauchy integral expression for $c_n ^s$, we will choose appropriate contours to bound the norms of the infinite tails. Then, we will control the central terms by analyzing the possibilities for those three coefficients.

First, let us focus on the positive tail. For the coefficients indexed by $n> 0$, we can adopt the contour $\gamma _\rho := \{ \rho e^{i\theta} |0\leq \theta \leq 2 \pi\}$ in their Cauchy integral expressions, valid for $\rho \in \left(\frac{1}{s}, s\right)$, to obtain:
\begin{align*}
c_n ^s & \equiv \oint _{\gamma _\rho} \frac{\phi ^s (\zeta)}{\zeta ^{n+1}} d|\zeta| \\
\Rightarrow |c_n^s| & \leq \oint _{\gamma _\rho} \frac{|\phi ^s (\zeta)|}{\rho ^{n+1}} dl
\end{align*}
Multiplying both sides by $2\rho$ and integrating over the interval $\rho\in (s-s^\ell,s)$, we see that:
\begin{align*}
\int _{s-s^\ell} ^s |c_n^s| \rho d\rho & \leq 2 \int _{s-s^\ell} ^{s} \oint _{\gamma_\rho} \frac{|\phi ^s (\zeta)|}{\rho ^{n}} \, dl \, d\rho\\
\Rightarrow |c_n^s| \left[ s^2 -\left(s-s^\ell\right)^2\right] & \leq \frac{2}{(s-s^\ell)^{n}} \int _{s-s^\ell} ^s \oint |\phi ^s(\zeta)| \, dl \, d\rho \\
\Rightarrow |c_n^s| & \leq \frac{2}{s \left(2 s^{\ell} -s ^{2\ell -1}\right)  (s-s^\ell)^{n}} \int _{s-s^\ell} ^s \oint |\phi ^s(\zeta)| \, dl \, d\rho
\end{align*}
valid for any $\ell\in \left(0,\frac{\log(s-1)}{\log(s)}\right)$. Note that the value $\ell = \frac{\log(s-1)}{\log(s)}$ yields $s-s^{\ell} = 1$, and the integral on the right hand side is exactly the $\mathcal{L}^1$-norm of $\Phi ^s$ on $\Omega_{1,s}$.

Thus, noting that $2|\phi ^s| \leq e^s$ and $E\left(u_s, \Omega _{s-s^{\ell},s}\right) \leq E\left(u_s, \Omega _{1,s}\right)  \leq E(w, \Omega _s) $, we establish
\begin{align*}
|c_n^s| & \leq \frac{2}{s \left(2 s^{\ell} -s ^{2\ell -1}\right)  (s-s^\ell)^{n}} \int _{s-s^{\ell}} ^s \oint |\phi ^s(\zeta)| \, dl \, d\rho \\
 & \leq \frac{1}{s \left(2 s^{\ell} -s ^{2\ell -1}\right)  (s-s^\ell)^{n}} E(u_s, \Omega _{s-s^\ell,s})\\
 & \leq \frac{1}{s \left(2 s^{\ell} -s ^{2\ell -1}\right)  (s-s^\ell)^{n}} E(u_s, \Omega _{1,s})\\
 & \leq \frac{1}{s \left(2 s^{\ell} -s ^{2\ell -1}\right)  (s-s^\ell)^{n}} E(w, \Omega _{1,s})
\end{align*}
The right hand side can be controlled because the energy density of the Scherk map $w$ obeys the estimate $$ e^w (z) \leq |z^{k-2}| + \mathcal{J}^w (z).$$
Here, $\int _{\C} \calJ ^w dA = 2\pi (k-2)$ since the area of $w(\C)$ is given by the area of an ideal $k$-gon. Thus,
\begin{align*}
|c_n^s| & \leq \frac{1}{s \left(2 s^{\ell} -s ^{2\ell -1}\right)  (s-s^\ell)^{n}} \left[ \int _{\Omega _s} \rho^{k-2} dA + 2\pi (k-2)\right] \\
& \leq \frac{1}{s \left(2 s^{\ell} -s ^{2\ell -1}\right)  (s-s^\ell)^{n}} \left[ \frac{2\pi }{k} s^{k} + 2\pi (k-2)\right]
\end{align*}
Note that $\lim_{s\nearrow +\infty} s^{1-\frac{log(s-1)}{log(s)}}=1$. Hence, there exists $D>0$ so that, for $s$ large enough, we have
\begin{align*}
|c_n^s| & \leq D s^{k-n-2}.
\end{align*}

This inequality can be used to obtain a bound on the norm of the positive tail on the core curve. We have:
\begin{align*}
\left| \sum _{n=2} ^{\infty} c^s_{nk-2} z^{nk-2} \right| & \leq \sum _{n=2} ^{\infty} \left| c^s_{nk-2} \right|\\
& \leq D \sum _{n=2} ^{\infty}  s^{k-[nk-2]-2} \\
& \leq D \sum _{n=2} ^{\infty}  s^{-k(n-1)} \\
& \leq D \frac{1}{1-\frac{1}{s^{k}}}
\end{align*}

Next, let's consider the negative tail. By Proposition \ref{prop:parachute-qd-identity}, since the parachute maps satisfy the reflectional symmetries $u\left(\frac{1}{\overline{z}}\right)=u(z)$ and $u(z)=\overline{u}(z)$, the Hopf differential satisfies the identity $$ \phi ^s (z) = \phi ^s \left(\frac{1}{z}\right) \frac{1}{z^4}.$$ Referring to the Laurent expansion for the Hopf differential, this reveals the identity
$$ c_{-(n+2)} ^s = c_{n-2} ^s$$
among the coefficients - relating them so that the sequence of coefficients is symmetric about the $n=-2$ index. Thus, the norm of the negative tail can be bounded along the core curve:
\begin{align*}
\left|\sum _{n=2} ^{\infty} \frac{c^s _{-nk-2}}{z^{nk+2}} \right| &\leq \sum _{n=2} ^{\infty} \left|c^s _{-nk-2} \right|\\
&\leq \sum _{n=2} ^{\infty} \left|c^s_{nk-2} \right|\\
& \leq D \frac{1}{1-\frac{1}{s^{k}}}
\end{align*}
as with the positive tail. Hence, for all $s$ large enough, for any $z\in \gamma _1$, we have bounded
$$ |\phi ^s (z)| \leq D+ \left| \frac{c^s_{-(k+2)}}{z^{k+2}} + \frac{c^s_{-2}}{z^2} + c_{k-2} ^s z^{k-2} \right|$$
where $D=D(k)$ is a positive constant given by the bound on the moduli of the tail coefficients.

Finally, let us turn to the central terms. We will relabel the coefficients as $C^s:=c^s_{-(k+2)}=c_{k-2} ^s$ and $D^s = c^s_{-2} $. Thus far, we've deduced that the central terms have the form
\begin{align*}
\frac{C^s}{z^{k+2}} + \frac{M^s}{z^2} + C^s z^{k-2} & = C^s \frac{1}{z^{k+2}} \left[ 1 + \frac{M^s}{C^s} z^{k} + z^{2k+2} \right]
\end{align*}
The symmetry $u_s (\overline{z}) = \overline{u_s (z)}$ ensures $C^s$ and $M^s$ are real for all $s$. Thus, along the unit circle, we have:
\begin{align*}
\left|\frac{C}{z^{k+2}} + \frac{M}{z^2} + C z^{k-2} \right| & \leq \left|\frac{C}{z^{k+2}}\right| + \left|\frac{M}{z^2}\right| + \left|C^s z^{k-2} \right|\\
& \leq 2|C^s| + |M^s|
\end{align*}
From this expression, it is clear that we would like to obtain a sequence $\{s_j\}$ for which both $|C^{s_j}|$ and $|M^{s_j}|$ are bounded uniformly. Thus, let $\{ s_j \}$ be any sequence tending to $+\infty$. From this sequence, we will extract a sub-sequence for which both $|C^{s_j}|$ and $|M^{s_j}|$ are bounded uniformly.

In the final analysis, it will be convenient to have both $C^{s_j}$ and $M^{s_j}$ non-vanishing for any $s_j$. To sufficiently reduce the analysis to this hypothesis, we will simply drop all the indices $s_j$ for which either $C^{s_j}$ or $M^{s_j}$ vanishes. This reduction requires us to delete either finitely many or infinitely many indices. If we are fortunate enough to require deleting infinitely many indices, we can instead extract a sub-sequence $\{s_l\}$ for which: either $C^{s_l}$ or $M^{s_l}$ vanish identically for all $s_l$.

Our process for producing the sub-sequence (and our plan of attack) is thus: produce a sub-sequence (also labeled $\{s_j\}$, by abuse of notation) along which $|C^{s_l}|$ and $|M^{s_l}|$ are bounded uniformly in each case.
\begin{enumerate}
\item[1.] If the reduction requires dropping infinitely many terms, then from those terms we can extract a subsequence $\{ s_j\}$ for which one of the following must hold:
    \begin{enumerate}
    \item for all $s_j$, we have $C^{s_j} = 0$ and $W^{s_j}= 0$; or
    \item for all $s_j$, we have $C^{s_j} = 0$ and $W^{s_j}\neq 0$; or
    \item for all $s_j$, we have $W^{s_j} = 0$ and $C^{s_j}\neq 0$
    \end{enumerate}
\item[2.] or else, if the reduction only requires dropping finitely many terms, we reduce the sequence so that
    \begin{itemize}
    \item[(d)] for all $s_j$, we have both $C^{s_j}\neq 0$ and $W^{s_j}\neq 0$
    \end{itemize}
\end{enumerate}

In case (a), with both coefficients vanishing identically, we trivially have $|\phi ^s(z)|$ uniformly bounded by $D(k)$ along the unit circle. For the (b) and (c) cases, we will argue by contradiction that there must be bounded subsequences of $|C^{s_j}|$ and $|M^{s_j}|$; we'll use two different geometric considerations to temper the norm of the non-vanishing coefficient. Finally, we will conclude the analysis with case (d) by studying the ratios of the coefficients and the underlying foliation structure of its Hopf differential.

In case (b), we have $C^{s_j}=0$, but $M^{s_j}\neq 0$. This suggests that having $|M^{s_j}|$ unbounded would lead to the distance between the core curve and the nearest zero in the $\Phi ^{s_j}$ metric becoming unbounded. In particular, this allows us to obtain the absurd statement: we can find a circle's worth of very large $\Phi ^{s_j}$-disks, on each of which $u^{s_j}$ converges to a map of $\C$ onto a geodesic in $\HH$, all while retaining its $k$-fold symmetry. Let us make this argument precise.

Suppose for a contradiction that $|M^{s_j}|$ is unbounded. Then there exists a subsequence $\{s_l\}$ for which $|M^{s_l}| \nearrow + \infty$. Consider the family of Hopf differentials $$\left\{ \frac{1}{M^{s_l}} \Phi ^{s_l}\right\}.$$ It forms a normal family (on compact subsets) since $$\left| \frac{1}{M^s} \Phi ^s \right| \leq \left|\frac{1}{z^2} \right| + \frac{1}{M^s} D(k)\left(|z| + \frac{1}{|z|} \right)$$
Thus, $\{ \frac{1}{M^{s_l}}\Phi ^{s_l} \}$ converges pointwise, and hence uniformly on compact subsets, to $\frac{1}{z^2}dz^2$ on $\C ^*$ in $s_l$. The zeros of $\{\frac{1}{M^{s_l}}\Phi ^{s_l}\}$ converge uniformly in $s_l$ to the zeros of $\frac{1}{z^2}dz^2$ - of which there are none. Hence, the original sequence $\{ \Phi ^{s_l}\}$ contains $\Phi ^{s_l}$-disks of size $M^{s_l}$.

Fix an angle $\theta$. We can find a fixed disk $Z(\theta) \subset \Omega _{1,s_l}$ so that $Z(\theta)$ has diameter at least $\frac{M^{s_l}}{2}$ in the $\Phi^{s_l}$-metric. In addition, we can choose it to lie away from the core curve $\gamma _1$ and with a center at an angle $\theta$ relative to the $x$-axis. By Proposition \ref{prop:parachute-maps_non-negative-J}, it follows that the Jacobian of $u_s$ is non-negative on this disk. Hence, the sequence of induced maps $$ (Z(\theta),d_{\Phi^{s_l}})\rightarrow \HH $$ converges to a map of $\C \rightarrow \HH$, since the sequence of metrics induced by $\Phi ^{s_l}$ on $Z(\theta)$ converges to the standard metric on $\C$.

From this procedure, for each $\theta$ we obtain a harmonic map from $\C$ to the hyperbolic plane. As each domain was obtained by using the $\Phi^{s_l}$-metric, the maps always have a constant (depending on $\theta$), non-zero Hopf differential. Here, we can apply Dumas-Wolf's uniqueness Theorem \ref{thm:dumas-wolf_uniqueness}: by interpreting the theorem as a statement yielding uniqueness of harmonic immersions from the plane into $\HH$, we know that each of these maps must be maps onto a geodesic.

Furthermore, we have a whole circle's worth of maps from $\C$ to $\HH$ whose images must be contained in the ideal $k$-sided polygon which bounds the images of $u^{s_l}$. In particular, we can choose $\theta$ to lie in the direction of one of the ideal corners of the polygon. However, we cannot obtain such a limiting map whose image is a geodesic perpendicular to this direction (as the horizontal foliation lies orthogonal to this direction) since such a geodesic cannot be contained in the regular ideal polygon. Hence, there cannot exist a subsequence $\{ s_l\}$ for which $|M^{s_l}| \nearrow + \infty$, and $\{ |M^{s_j}|\}$ must contain a bounded subsequence.

In case (c), we have $M^{s_j}=0$, but $C^{s_j}\neq 0$. This suggests that having $|C^{s_j}|$ unbounded would lead to too much energy (when compared to the Scherk map). Let us make this precise. Suppose for a contractiction that $\{|C^{s_j}|\}$ does not contain a bounded subsequence, so that it contains a subsequence with $\{|C^{s_l}|\} \nearrow + \infty$ monotonically. Observe, then, that
\begin{align*}
|\Phi ^{s_l}(z) | & \geq \left|C^{s_l} z^{k-2}\right| - \left|\frac{C^{s_l}}{z^{k+2}}\right| - D(k) \left( \rho + \frac{1}{\rho} \right) \\
& \geq |C^{s_l}| \left( \rho ^{k-2} - \frac{1}{\rho ^{k+2}} \right) - D(k) \left( \rho + \frac{1}{\rho} \right)
\end{align*}
Thus, it follows that the energy density of $u_s$ can be bounded from below by
\begin{align*}
e^{s_l}(z) & \geq |\Phi ^{s_l}(z) |\\
 & \geq |C^{s_l}| \left( \rho ^{k-2} - \frac{1}{\rho ^{k+2}} \right) - D(k) \left( \rho + \frac{1}{\rho} \right)
\end{align*}
where $D(k)$ is the constant from the proof of Proposition \ref{prop:bounded-qd-subseqn}.

This suggests that the energy density is growing at a rate of $\rho ^{k-2}$, with a proportional constant $\left| C^{s_l}\right|$. However, comparing to the Scherk map, we necessarily have
\begin{align*}
E(\Omega_{1,s}, u_s) &\leq E(\Omega_{1,s}, w) +  E(\Omega_{1}, w)\\
\Rightarrow E(\Omega_{1,s}, u_s) &\leq E(\Omega_{1,s}, w) +  T
\end{align*}
where $T:= E(w , \Omega_{1})$. Furthermore, the energy density of the Scherk map satisfies
$$ e^w(z) = \rho ^{k-2} + \calJ (z) $$
for which $\calJ ^w > 0$ and $ \int _{\Omega _{1,s}} \calJ \, dA \leq 2\pi k$.
Note that the energy estimate is supposed to hold on the domain $\Omega _{1,s}$ Hence, for all $s>1$,
\begin{align*}
\int_{\Omega_{1,s}} e^w(z) dA- \int_{\Omega_{1,s}} e^s(z) dA+ T \geq 0
\end{align*}
\begin{align*}
\Rightarrow \int_{\Omega_{1,s}} \rho ^{k-2} \, dA + 2\pi k - \int_{\Omega_{1,s}} \left[|C^{s_l}| \left( \rho ^{k-2} - \frac{1}{\rho ^{k+2}} \right)- D(k) \left( \rho + \frac{1}{\rho} \right)\right] \,dA + T \geq 0\\
\Rightarrow \int_{\Omega_{1,s}} \left[\left( 1 - |C^{s_l}| \right) \rho ^{k-2} - |C^{s_l}| \frac{1}{\rho^{k+2}} -\int_{\Omega_{1,s}} D(k) \left( \rho + \frac{1}{\rho} \right) \right] dA \geq -2\pi k - T
\end{align*}
The right hand side is a constant yet, for $s$ large enough, the left hand side becomes arbitrarily large and negative since $|C^{s_l}|\nearrow +\infty$. This is a contradiction. Hence, there cannot exist a subsequence with $|C^{s_l}|\nearrow +\infty$. In other words, when the original subsequence $\{C^{s_j}\}$ falls into case (c), we can choose a subsequence of it for which $\{|C^{s_l}|\}$ stays bounded.

Finally, we are left with considering the fourth possibility, for which $C^{s_j}\neq 0$ and $W^{s_j}\neq 0$ for all $s_j$. In this situation, we can consider the sequence of norms of the ratios of the coefficients. In $\widehat{\R _{\geq 0}}= \R _{\geq 0} \cup \{ + \infty \}$, the possibilities are:
\begin{enumerate}
\item[(e)] there exists a sequence $\{ s_j \}\nearrow +\infty$ along which $\left|\frac{C^{s_j}}{M^{s_j}}\right| \nearrow +\infty$, or
\item[(f)] there exists a sequence $\{ s_j \}\nearrow +\infty$ along which $\left|\frac{M^{s_j}}{C^{s_j}}\right| \nearrow +\infty$, or
\item[(g)] there exists a sequence $\{ s_j \}\nearrow +\infty$ along which $\left|\frac{C^{s_j}}{M^{s_j}}\right| \rightarrow D \in \R_{+}$
\end{enumerate}
We will proceed by contradiction to rule out the possibility of cases (e) and (f). The flavor of these arguments will resemble the purely vanishing cases (b) and (c) above. The resolution of (g) will follow from the fact that, in this case, either {\em both} $|C^{s_j}|$ and $|M^{s_J}|$ diverge or {\em both} stay bounded. The last step is to rule out their simultaneous divergence, again by finding a circle's worth of large $\Phi ^s$-disks as in (b).

In case (e), there exists a sequence $\{ s_j \}\nearrow +\infty$ along which $\left|\frac{C^{s_j}}{M^{s_j}}\right| \nearrow +\infty$. Then there exists a subsequence $\{s_l\}$ so that $|M^{s_l}| < |C^{s_l}|$ and $|C^{s_l}|\nearrow + \infty$. We employ an energy estimate similar to that in case (c):
\begin{align*}
|\Phi ^{s_l}(z) | & \geq \left|C^{s_j} z^{k-2}\right| - \left|\frac{C^{s_j}}{z^{k+2}}\right| - \left|\frac{M^{s_j}}{z^2}\right| - D(k) \left( \rho + \frac{1}{\rho} \right) \\
& \geq |C^{s_l}| \left( \rho ^{k-2} - \frac{1}{\rho ^{k+2}} \right) - |M^{s_l}|\frac{1}{\rho ^2} - D(k) \left( \rho + \frac{1}{\rho} \right)\\
& \geq |C^{s_l}| \left( \rho ^{k-2} - \frac{1}{\rho ^{k+2}} \right) - |C^{s_l}|\frac{1}{\rho ^2}- D(k) \left( \rho + \frac{1}{\rho} \right)
\end{align*}
\begin{align*}
\Rightarrow e^{s_l}(z) \geq |C^{s_l}| \left( \rho ^{k-2} - \frac{1}{\rho ^{k+2}}  - \frac{1}{\rho ^2}\right) - D(k) \left( \rho + \frac{1}{\rho} \right)
\end{align*}
At this point, we reach the same conclusion as in case (c) by contradicting the necessary energy comparison $E(u^{s_l},\Omega_{1,s_l}) \leq E(w,\Omega_{s_l})$. Hence, no subsequence $\{s_j\}$ can be extracted which falls into case (e).

In case (f), there exists a sequence $\{ s_j \}\nearrow +\infty$ along which $\left|\frac{M^{s_j}}{C^{s_j}}\right| \nearrow +\infty$. Then there exists a subsequence $\{s_l\}$ so that $|C^{s_l}| < |M^{s_l}|$ and $|M^{s_l}|\nearrow + \infty$. As in case (b), we can consider the family of holomorphic quadratic differentials
$$ \left\{ \frac{1}{M^{s_l}} \Phi^{s_l} \right\}.$$
Since $\left|\frac{M^{s_j}}{C^{s_j}}\right| \nearrow +\infty$, we obviously have the reciprocal satisfying $\left|\frac{C^{s_j}}{M^{s_j}}\right| \searrow 0$. Hence,
$$ |\Phi^{s_l}(z)| \leq \left| \frac{1}{z^2}\right| + \left|\frac{C^{s_l}}{M^{s_l}}\right|\left(z^{k-2} + \frac{1}{z^{k+2}} \right) + \frac{1}{M^{s_l}} D(k)\left( |z| +\frac{1}{|z|}\right)$$
This implies the family $\left\{ \frac{1}{M^{s_l}} \Phi^{s_l} \right\}$ constitutes a normal family on compact subsets. Thus, $\{ \frac{1}{M^{s_l}}\Phi ^{s_l} \}$ converges pointwise, and hence uniformly on compact subsets, to $\frac{1}{z^2}dz^2$ on $\C ^*$ in $s_l$. At this point, we reach the same conclusion as in case (b) by contradicting the existence of a circle's worth of maps to geodesics. Hence, no subsequence $\{s_j\}$ can be extracted which falls into case (f).

In case (g), there exists a sequence $\{ s_j \}\nearrow +\infty$ along which $\left|\tfrac{C^{s_j}}{M^{s_j}}\right| \rightarrow D \in \R_{+}$. Suppose for a contradiction that $|C^{s_j}|$ and $|M^{s_j}|$ simultaneously diverge.
Consider the family of holomorphic quadratic differentials
$$ \left\{ \frac{1}{M^{s_l}} \Phi^{s_l} \right\}$$
on $\Omega_{1,s}$. Observe that we have the bound
$$ \left| \frac{1}{C^{s_j}} \Phi^{s_j} (z)\right| \leq \left| z^{k-2} + \frac{C^{s_j}}{M^{s_j}}\frac{1}{z^2} + \frac{1}{z^{k+2}}\right| + \frac{1}{C^{s_j}}D(k)\left(|z|+\frac{1}{|z|}\right).$$
Hence, the family converges pointwise (and uniformly on compact subsets) to
$$ \left( z^{k-2} + D\frac{1}{z^2} + \frac{1}{z^{k+2}}\right) dz^2$$
on $\Omega_{1,s}$. Finally, observe that the coefficient has the form
$$ \frac{1}{z^{k+2}}\left( z^{k} -N\right) \left( z^{k} - \frac{1}{N}\right) dz^2$$
where
$$ N + \frac{1}{N} = D$$
Hence, in the $\Phi ^{s_l}$ metric, we can find larger $\Phi ^{s_l}$-disks away from the core curve. Again, we can find a circle's worth of maps to geodesics (like in case (b)) and obtain a contradiction. Thus, $|C^{s_j}|$ and $|M^{s_j}|$ cannot simultaneously diverge. Since we have $\left|\frac{C^{s_j}}{M^{s_j}}\right| \rightarrow D \in \R_{+}$, it must be that $|C^{s_j}|$ and $|M^{s_j}|$ both stay bounded.

\end{proof}

\subsection{Bounded Core Lemma}\label{section:toyproblem-bounded_core_lemma}
We will call the curve $\gamma _1$ the core curve, since it is fixed by the reflectional symmetry on the domain $\Omega_{\frac{1}{s},s}$. Note that it is the boundary component of $\Omega_{1,s}$ on which the harmonic mapping problem \eqref{eqn:PFDrs} prescribes a free boundary.

Here, we state the Bounded Core Lemma: Consider the Scherk map restriction $$w:\Omega \rightarrow \HH.$$ If we ``poke a hole of size $\Omega _r$" in the domain $\Omega _s$, and flow the harmonic map on its changed domain to allow the hole to move freely, the Bounded Core Lemma purports that the hole does not move very far from the origin $\mathcal{O}\in \HH$, even though the hole itself may deform.

\begin{lemma} \label{lem:boundedcore}
{\em (Bounded Core Lemma)} Consider the parachute map $u_s : \Omega _{1,s} \rightarrow \HH$ as above. Let $\gamma _1:=\{|z|=1 \}$ denote the core curve and $\mathcal{O}$ denote the origin in the Poincare disk model for $\HH$. Then there exists a sequence of indices $\{s_j\in \R_{\geq 0 }\}_j$ with $s_j \rightarrow \infty$ so that we have $$d(\mathcal{O} , u^{s_j}(z)) \leq M$$ for any $z\in \gamma _1$, for some $M = M(k)$ depending only on the degree $k$ of the Scherk map's Hopf differential.
\end{lemma}

\begin{proposition}\label{prop:core-energy-bounds-dist}
For each parachute map $u_s$ defined as above, there exists a constant $B>0$ independently of $s$ so that $$d(\mathcal{O} , u_s(z)) \leq B \cdot \left[ \sup_{z\in \gamma _1} e^s (z) \right] ^{\frac{1}{2}}$$
\end{proposition}
\begin{proof}
The image of the core curve $\gamma _1$ under the map $\widetilde{u_s}$ inherits the same symmetries of the map $\widetilde{u_s}$, so $$d(\mathcal{O} , u_s(z)) \leq L (u_s (\gamma _1))$$ for all $z\in \gamma _1$, where $\mathcal{O}$ is the origin of the hyperbolic plane and $L(u_s(\gamma _1))$ denotes the length of the image of the core curve under $u_s$. Hence, it suffices to bound $L(u_s(\gamma _1))$ for all $s$.
Note that
\begin{align*}
L(u_s(\gamma _1)) & \equiv \int _{\gamma _1} |Du ^s (\partial _ \theta )| dl \\
  & \leq (2\pi) ^\frac{1}{2} \left(\int _{\gamma _1} |Du(\partial _ \theta )|^2 dl \right)^{\frac{1}{2}}\\
  & \leq (2\pi) ^\frac{1}{2} \left(\int _{\gamma _1} ||Du||^2 dl \right)^{\frac{1}{2}}\\
  & \leq (2\pi) ^\frac{1}{2} \left(\int _{\gamma _1} e^s(z) dl \right)^{\frac{1}{2}}\\
  & \leq B \cdot \left[ \sup_{z\in \gamma _1} e^s (z) \right] ^{\frac{1}{2}}
\end{align*}
where $B=(2\pi) ^\frac{1}{2} \cdot L(\gamma _1)^{\frac{1}{2}}$ is a constant independent of $s$.
\end{proof}
Note that the initial rescaling of $\Omega _{r,s}$ by a factor of $\frac{1}{r}$ changes $L(\gamma _1)$ by a factor of $r^{\frac{1}{2}}$. This only changes the constant appearing on the right-hand side of Proposition \ref{prop:core-energy-bounds-dist}.

\begin{proposition}\label{prop:qd-equals-core-energy}
For each parachute map $u_s$ defined as above, for each $z\in \gamma _1$, we have $$e^s(z) = 2|\phi ^s(z)|.$$
\end{proposition}

\begin{proof}
In general, we have an inequality relating the energy density, the Hopf differential, and the Jacobian:
$$ 2||\Phi (z) || \leq e (z) \leq 2||\Phi (z)||  + | \mathcal{J} (z)|.$$
These follow from the equations for the holomorphic and anti-holomorphic energies:
\begin{align*}
e & = \calH + \calL \\
\calJ & = \calH - \calL \\
||\Phi || & = \sqrt{ \calH \calL }
\end{align*}

Since the map $u_s$ is symmetric by reflection across the core curve, the identity $\mathcal{J}^s (z) = 0$ holds for all $z\in \gamma _1$ and for all $s$. Thus, $e^s(z) = 2 ||\Phi ^s(z)||$. Furthermore, since $\sigma(z) \equiv 1$, we have simply $$e^s(z) = 2|\phi ^s(z)|.$$
\end{proof}

We now have enough to assemble the proof of the Bounded Core Lemma \ref{lem:boundedcore}.
\begin{proof}
We aim to show that there exists $M>0$ such that, for all $s$ large enough, we have
$$d(\mathcal{O} , u_s(z)) < M.$$

By Proposition \ref{prop:core-energy-bounds-dist}, each parachute map $u_s$ satisfies $$d(\mathcal{O} , u_s(z)) \leq B \cdot \left[ \sup_{z\in \gamma _1} e^s (z)\right] ^{\frac{1}{2}}.$$

By Proposition \ref{prop:qd-equals-core-energy}, for each $z\in \gamma _1$, we have $$e^s(z) = 2|\phi ^s(z)|.$$

By Proposition \ref{prop:bounded-qd-subseqn}, for $r$ large enough, for all $s>r$, there exists a positive constant $M < \infty$ such that, for all $z\in \gamma _1$, we have  $$|\phi ^s (z)|<M.$$

Combining these, we arrive at our desired conclusion.
\end{proof}

\subsection{The Energy Estimate: Parachute maps tied to Scherk map}\label{section:toyproblem-the_energy_estimate}
With the Bounded Core Lemma \ref{lem:boundedcore} in hand, we now derive an estimate on the discrepancy $E(w,\Omega_{r,s}) - E(u_s , \Omega_{r,s})$ between the energies of the parachute map $u_s$ and the Scherk map $w$ over their common domains.
\begin{lemma}\label{lem:energyestimate}
(The Energy Estimate)
There is a function $F(r)$ such that, when $r$ is large enough, for all $s>r$, the following inequality holds:
$$E(\Omega_{r,s},w) - E(\Omega_{r,s},u_s) < F(r)$$
\end{lemma}
\begin{proof}
We can estimate the discrepancy between the energies of $w|_{\Omega_{r,s}}$ and $u_s|_{ \Omega_{r,s}}$ by the energy of a harmonic map of a disc $f_s: \Omega_s \rightarrow \HH$. To do this, consider the unique harmonic map $p_s : \Omega _r \rightarrow \HH$ which minimizes energy in its homotopy class of maps with boundary value $p_s | _{\gamma_r} \equiv u_s |_{\gamma _r}$. Define the piecewise harmonic map $f _s :\Omega _s \rightarrow \HH$ by
\begin{align*}
f_s |_{\Omega _r}&= p_s\\
f_s |_{\Omega_{r,s}}&=u_s.
\end{align*}
The map $f_s$ is such that $E(\Omega _s, w) \leq E(\Omega _s, f_s)$, since $w$ is actually smooth and energy minimizing in all of $\Omega _r$. Since we can decompose these energies over the domain $\Omega _s = \Omega _r \cup \Omega _{r,s}$ as
\begin{align*}
E(\Omega_s, w) &= E(\Omega_r, w) + E(\Omega_{r,s}, w)\\
E(\Omega _s, f_s) &= E(\Omega _r, p_s) + E(\Omega_{r,s}, u_s)
\end{align*}
their difference on $\Omega _{r,s}$ can be bounded as
$$E(\Omega_{r,s}, w) - E(\Omega_{r,s}, u_s) \leq E(\Omega_r, p_s) - L ,$$
where the constant $L=L(k,r)$ is simply the energy of the Scherk map on the disk $\Omega _r$:
$L(k,r):= (\Omega _r , w).$
Hence, it suffices to exhibit a bound on $E(\Omega _r, p_s)$ in terms of $r$ alone. To do this, we will use the Bounded Core Lemma \ref{lem:boundedcore} in addition to the iso-energy inequality (Theorem \ref{thm:isoenergy}) and Cheng's interior gradient estimate (Theorem \ref{thm:chenglemma}). 
We apply the iso-energy inequality (Theorem \ref{thm:isoenergy}) first. Since $\Omega _r$ is a ball, the iso-energy inequality allows us to conclude
$$ E(\Omega _r, p_s) \leq E(\gamma _r, p_s| _{\partial \Delta r})$$
for which we know the right-hand side is uniformly bounded.
Recalling that
$$E(\Omega_{r,s}, w) - E(\Omega_{r,s}, u_s) \leq E(\Delta_r, p_s) - K_r ,$$
we also arrive at $E(\Omega_{r,s}, w) - E(\Omega_{r,s}, u_s)$ being uniformly bounded in $s$ for all $s>r$.
Let us now specialize Cheng's interior gradient estimate (Theorem \ref{thm:chenglemma}) to our situation. We use $M=\Omega_{r,s}$ and $N=\HH$, so that we can take $K=0$. A choice for the points $x_0$, $y_0$ and the constants $a$, $b$ will be chosen momentarily. We are guided by the following application of the Bounded Core Lemma.

Fix $r$ for which there is a sequence of indices $s_j \nearrow \infty$ as in the Bounded Core Lemma \ref{lem:boundedcore}. It follows that $$d(\mathcal{O} , u^{s_j}(z)) \leq M$$
for all $s_j$. By the maximum principle for the (subharmonic) distance function between the harmonic Scherk map $w$ and the harmonic parachute map $u_s$, we have that
\begin{align*}
\max_{p\in \Omega_{r,s}} \{d_\HH (u_s (p) , w(p) )\} &\leq \max_{p\in \gamma _r \cup \gamma _s}\{d_\HH (u_s (p) , w(p) )\} \\
& \leq \max_{p\in \gamma _r}\{d_\HH (u_s (p) , w(p) )\}
\end{align*}
for all $s >r$. The second inequality follows from the pointwise agreement of $u_s$ with $w$ on $\gamma _s$. Applying the triangle inequality, we obtain
\begin{align*}
\max_{ p\in \Omega_{r,s}}\{d_\HH (u_s (p) , w(p) )\} &\leq \max_{p\in \gamma _r}\{d_\HH (u_s (p) , \mathcal{O} )\} + \max_{p\in \gamma _r}\{d_\HH (w (p) , \mathcal{O} )\}\\
&\leq M + \max_{p\in \gamma _r}\{d_\HH (w (p) , \mathcal{O} )\}\\
&\leq M + P
\end{align*}
where $P=P(k)$ is a constant defined as
$$P(k) := \max_{p\in \gamma _r}\{d_\HH (w (p) , \mathcal{O} )\}< \infty.$$

Fix a constant $a>r$, choose $x_0 \in \gamma_r$ fix any point $y_0\in \HH$ at a distance $R$ of at least $M+P$ from the set $w(\Delta _r)$. Then there exist constants $b>1$ and $\beta >0$ are some independent of $s_j$, all of which are finite (although their admissible values are dependent on $k$ and the map $w$ on $\Delta _r)$). Choosing such $b$ and $\beta$, Cheng's Lemma \ref{thm:chenglemma} states that

$$|\nabla u_s (x_0)| ^2\leq c_m \frac{b^2(b^2 - \rho ^2 \circ u_s) ^2 }{a^2\beta}$$
Because $y_0$ is chosen to be a distance $R$ from $w(\Omega _r)$, for each $p \in \Omega _{r,s}$, we have:
\begin{align*}
\rho \circ u^s (p) &\equiv d_{\HH} (u^s (p), y_0)\\
& \leq d_{\HH} (u^s (p), w(p)) + d_{\HH} (w(p),y_0 )\\
& \leq (M + P) + R
\end{align*}
Hence it follows that $ \rho ^2 \circ u_s (x_0)\leq (M+P+R)^2$ and that
$$ b^2 - \rho ^2 \circ u_s \leq  b^2 + (M+P+R)^2.$$
This shows that $|\nabla u_s (x_0)| ^2$ is uniformly bounded for all $x_0\in \gamma _r$, for all $s_j$.

Finally, observe that $\nabla u_s$ has zero $\partial _n$ component along $\gamma _r$, so that $$|\nabla u_s|(x_0) = |\nabla u_s (\partial _\theta)|(x_0) $$ for all $x_0\in \gamma _r$. Since $p_s |_{\gamma _r} \equiv u_s|_{\gamma_r}$ is the parameterized Dirichlet boundary condition for $p_s$, it is thus true that $$|\nabla p_s (\partial _\theta)|(x_0)  = |\nabla u_s |(x_0) .$$ As we have shown that $|\nabla u_s (x_0)| ^2$ is uniformly bounded for all $x_0\in \gamma _r$, for all $s_j$, it follows that $E(\partial \Omega _r, p_s | _{\partial \Omega _r}),$ and hence $E(\Omega_r , p_s) $ is bounded independently of $s$.
\end{proof}


\section{Concluding remarks}

We return to addressing the original investigation of harmonic maps between compact hyperbolic surfaces. Theorem \ref{thm:crusher_existence} describes a harmonic map locally around a handle collapse. We describe below two procedures which suggest the maps from Theorem \ref{thm:crusher_existence} are indeed the local model for handle crushing harmonic maps between \emph{compact} surfaces, in that they approximate well a neighborhood of the handle collapsing.

One process considers a sequence of handle crushing harmonic maps $h_t :(\Sigma _g, \sigma_t) \rightarrow (\Sigma _{h<g}, \rho)$ between compact surfaces. Adopting an appropriate sense for a limiting harmonic map, we take $t\rightarrow \infty$ and obtain a harmonic map $h:(\Sigma _g ^*, \sigma) \rightarrow (\Sigma _h,\rho)$ from a \emph{noded} surface $(\Sigma _g ^* , \sigma) = (\Sigma_{h} ^* , \sigma) \vee (\Sigma _{g-h} ^*, \sigma)$. We expect to recover (away from the node point) a harmonic diffeomorphism $h:(\Sigma_{h} ^* ,\sigma) \rightarrow (\Sigma _h,\rho)$ on one part and a handle crushing map $h:(\Sigma_{g-h} ^*,\sigma) \rightarrow (\Sigma_h ,\rho)$ (of the type we have produced in Theorem \ref{thm:crusher_existence}) on the other part. Algebraically, we describe this as a decomposition of the domain on which the limiting harmonic map induces an isomorphism or maps trivially on fundamental groups.

The second process suggests our non-compact handle crushing map can be built up as part of an almost harmonic map, and a singular perturbation argument (or more specifically, a bridge principle) can be employed to correct it to become harmonic. We expect the correction to be small, so that our model is a close approximate on the handle crushing component. This correction is expected to be small by relating the data of the Hopf differentials near the bridged components.

We describe these two processes in more detail, below.

\vskip 11pt

\emph{Limit of compact maps to a diffeomorphism and a handle crushing model:} Let us first describe an example of a family of holomorphic quadratic differentials $\Phi _t$ on a genus two surface which limits to a holomorphic quadratic differential $\Phi$ on a noded surface of genus two (which is topologically a union of two punctured genus one surfaces away from the node point). This example is borrowed from $\sec A.4$ of \cite{McMullen89}.

Fix a torus with holomorphic quadratic differential $dz^2$. Cut open two segments of its horizontal foliation, each of length $L$ and away from the singularities, and glue in a Euclidean cylinder of height $H$ and circumference $t:=2L$ with its foliation by circles. This results in a holomorphic quadratic differential $\Phi _t$ on a genus two surface with four zeros at the endpoints of the slits. Suppose the slits are the sides of a square and let $t\rightarrow 0$ while $\frac{H}{L}$ is fixed. Fixing a basepoint, the limit quadratic differential $\Phi$ lives on the punctured torus with a fourth order pole.

Let us adopt this notion of a limiting pair $(\Sigma _g, \Phi _t)$ of Riemann surface $\Sigma _g$ with holomorphic quadratic differential $\Phi _t$ in the following construction. Let $h_t:(\Sigma _g,\sigma _t) \rightarrow (\Sigma _{g-1}, \rho)$ be a sequence of handle crushing harmonic maps in which $\sigma _t$ is an unbounded ray in the Teichm\"uller metric on $\mathcal{T}_g$. Suppose $h_t$ induces the same homomorphism on fundamental groups for all $t$. Note that a single handle is being crushed for each $t$. Fixing a basepoint in a neighborhood of the crushed handle, we consider the limit of $(\Sigma _g, \Phi ^{h_t})$, and ask:

\begin{question}
Can this process be carried out to produce our handle crushing map of a square punctured torus onto an ideal square from Corollary \ref{cor:torus_crusher}?
\end{question}

\emph{Bridging harmonic diffeomorphisms with handle crushers:} A \emph{bridge principle} would be effective in showing that the harmonic non-compact handle crushing model from Theorem \ref{thm:crusher_existence} closely describes the harmonic handle crushing near a neighborhood of a handle crushing map between compact surfaces. Let us describe our reasoning for expecting this.

Let $(\Sigma _h, \sigma)$ be a compact Riemann surface, and choose a holomorphic quadratic differential $\Phi$ which has a zero $p\in \Sigma _h$ of order $k-2$. Consider the one-parameter family of holomorphic quadratic differentials $$ \Phi ^t := t\Phi$$ obtained by scaling $\Phi$. For each value of $t$, there exists a unique hyperbolic metric $\rho _t$ on $\Sigma _h$ for which there exists a unique harmonic diffeomorphism $$u_t \equiv u_{\rho ^t} : (\Sigma _h , \sigma) \rightarrow (\Sigma _h, \rho _t )$$ homotopic to the identity $id_{\sigma, \rho}:(\Sigma _h , \sigma) \rightarrow (\Sigma _h, \rho _t ),$ by work of Wolf \cite{Wolf89}.

Observe that there is a sequence of $\Phi ^t$-disks $D^t$ centered at $p$ of increasing and diverging radii $r^t = 2 diam(D^t)$. This produces a sequence of harmonic maps $$ u^t: D^t \rightarrow \HH$$ with Hopf differentials $\Phi(u^t) = t\Phi$ (by construction). Hence, upon rescaling, we have a sequence of harmonic maps: $$ v^t: \sqrt{t}\cdot D^t \rightarrow \HH$$ with Hopf differentials $\Phi(v^t) = \Phi$. This implies that $v^t$ converges to a harmonic map $v:\C \rightarrow \HH$ with a Hopf differential $\Phi ^v$ with a zero of order $k-2$ at the origin. We view $(\C, \Phi ^v)$ as the punctured sphere with a Hopf differential having a pole of order $k+2$.

The foliation near the pole of $\Phi ^v$ is similar to the foliation near the pole of a meromorphic quadratic differential on a punctured Riemann surface $\Sigma _{g-h} ^*$ of positive genus. If $k$ is even, Theorem \ref{thm:crusher_existence} states there is a harmonic map $w:\Sigma _{g-h} ^* \rightarrow \mathcal{P}_k$.

Now, consider the harmonic diffeomorphisms $u^t : (\Sigma _h , \sigma) \rightarrow (\Sigma _h, \rho _t )$ above and the handle crushing harmonic map $w:\Sigma _{g-h} ^* \rightarrow \mathcal{P}_k$. Near the zero of order $k-2$ of $\Phi ^{u_t}$ and near the pole of $\Phi ^w$, the (unmeasured) foliations look more similar as $t\nearrow \infty$. Guided by these similar foliations, we can cut appropriately and bridge these maps together to produce an almost everywhere harmonic map. We ask:

\begin{question}
Can we bridge the harmonic map $u^t : (\Sigma _h , \sigma) \rightarrow (\Sigma _h, \rho _t )$ (for some $t$) with the harmonic handle crushing map $w:\Sigma _{g-h} ^* \rightarrow \mathcal{P}_k$ to produce a harmonic map $f:\Sigma _g \rightarrow \Sigma _h$ between compact Riemann surfaces?
\end{question}

If this is possible, the Hopf differential data and the diffeo-topological singularities of the map $f$ can be studied through our handle crushing models.

\bibliographystyle{alpha}
\bibliography{harmonic-map_puncRS-H2poly}

\end{document}